\documentclass{amsart}

\usepackage[OT2, T1]{fontenc}
\usepackage{url}
\usepackage{amsmath}
\usepackage{array}
\usepackage{graphicx}
\usepackage{amsfonts}
\usepackage{amssymb}
\usepackage{amstext}
\usepackage{amsthm}
\usepackage{hyperref}
\usepackage{colonequals}
\usepackage{enumitem}
\usepackage[alphabetic,lite]{amsrefs}
\usepackage{cleveref}
\usepackage[all,cmtip]{xy}
\usepackage{fullpage}

\numberwithin{equation}{subsection}

\newtheorem{theorem}{Theorem}
\numberwithin{theorem}{section}
\newtheorem{lemma}[theorem]{Lemma}
\newtheorem{proposition}[theorem]{Proposition}
\newtheorem{corollary}[theorem]{Corollary}

\newtheorem{conj}[theorem]{Conjecture}

\theoremstyle{definition}
\newtheorem{defn}[theorem]{Definition}

\theoremstyle{remark}
\newtheorem{rem}[theorem]{Remark}
\newtheorem*{claim}{Claim}
\newtheorem{para}[theorem]{}

\newcommand{\bC}{\mathbb{C}}

\newcommand{\bF}{\mathbb{F}}
\newcommand{\bG}{\mathbb{G}}
\newcommand{\bH}{\mathbb{H}}
\newcommand{\bI}{\mathbb{I}}
\newcommand{\bN}{\mathbb{N}}
\newcommand{\bP}{\mathbb{P}}
\newcommand{\bQ}{\mathbb{Q}}
\newcommand{\bR}{\mathbb{R}}

\newcommand{\bZ}{\mathbb{Z}}

\newcommand{\cF}{\mathcal{F}}

\newcommand{\cH}{\mathcal{H}}

\newcommand{\cL}{\mathcal{L}}

\newcommand{\cN}{\mathcal{N}}
\newcommand{\cO}{\mathcal{O}}
\newcommand{\cP}{\mathcal{P}}

\newcommand{\cT}{\mathcal{T}}

\newcommand{\cX}{\mathcal{X}}

\newcommand{\fp}{\mathfrak{p}}

\newcommand{\fsp}{\mathfrak{s}\mathfrak{p}}

\newcommand{\dA}{A^\vee}
\newcommand{\dB}{B^\vee}
\newcommand{\et}{{\text{\'et}}}
\newcommand{\gs}{{\gamma_\sigma}}
\newcommand{\rs}{{\rho_\sigma}}
\newcommand{\ob}{\overline{\beta}}
\newcommand{\ub}{\underline{\beta}}

\newcommand{\dR}{_{\mathrm{dR}}}

\newcommand{\cris}{_{\mathrm{cris}}}
\newcommand{\MT}{_{\mathrm{MT}}}
\newcommand{\aT}{{\mathrm{aT}}}

\DeclareMathOperator{\GL}{GL}

\DeclareMathOperator{\GSp}{GSp}
\DeclareMathOperator{\Sp}{Sp}
\DeclareMathOperator{\SO}{SO}

\DeclareMathOperator{\Gal}{Gal}
\DeclareMathOperator{\End}{End}
\DeclareMathOperator{\Hom}{Hom}

\DeclareMathOperator{\Lie}{Lie}
\DeclareMathOperator{\Res}{Res}
\DeclareMathOperator{\Spec}{Spec}
\DeclareMathOperator{\ad}{ad}
\DeclareMathOperator{\Span}{Span}
\DeclareMathOperator{\der}{der}
\DeclareMathOperator{\Fil}{Fil}
\DeclareMathOperator{\Gr}{Gr}

\begin{document}

\title[Cycles in the de Rham cohomology of abelian varieties]{Cycles in the de Rham cohomology of abelian varieties over number fields}

\author {Yunqing Tang}

\address{Department of Mathematics, Princeton University, Fine Hall, Washington Road\\
Princeton, New Jersey 08540, USA}
\email{yunqingt@math.princeton.edu} 

\begin{abstract}
In his 1982 paper, Ogus defined a class of cycles in the de Rham cohomology of smooth proper varieties over number fields. This notion is a crystalline analogue of $\ell$-adic Tate cycles. In the case of abelian varieties, this class includes all the Hodge cycles by the work of Deligne, Ogus, and Blasius. Ogus predicted that such cycles coincide with Hodge cycles for abelian varieties. In this paper, we confirm Ogus' prediction for some families of abelian varieties. These families include geometrically simple abelian varieties of prime dimension that have nontrivial endomorphism ring. The proof uses a crystalline analogue of Faltings' isogeny theorem due to Bost and the known cases of the Mumford--Tate conjecture.
\end{abstract}
\maketitle

\section{Introduction} 
\subsection{The Mumford--Tate conjecture and a conjecture of Ogus}\label{intro_Ogus}
Given a polarized abelian variety $A$ over a number field $K$, the Mumford--Tate conjecture for $A$ concerns the $\ell$-adic algebraic monodromy group $G_\ell$, which is defined to be the Zariski closure of the image of the Galois representation of $\Gal(\bar{K}/K)$ on the $\ell$-adic Tate module of $A$. The conjecture predicts that the connected component $G_\ell^\circ$ of $G_\ell$ coincides with the base change $G\MT\otimes \bQ_\ell$ of the Mumford--Tate group $G\MT$ of $A$. In terms of cycles, the conjecture asserts that $\ell$-adic Tate cycles coincide with $\bQ_\ell$-linear combinations of Hodge cycles. Here elements in the tensor algebra of the $\bQ$-coefficient Betti cohomology are called \emph{Hodge cycles} if they lie in the $(0,0)$-piece of the Hodge decomposition; elements in the tensor algebra of the $\ell$-adic \'etale cohomology are called \emph{$\ell$-adic Tate cycles} if they are, after a suitable Tate twist, fixed by $\Gal(\bar{L}/L)$ for some finite extension $L$ of $K$.

As a crystalline analogue, Ogus defined absolute Tate cycles for any smooth projective variety $X$ over a number field $K$ and predicted that, via the de Rham--Betti comparison, absolute Tate cycles coincide with absolute Hodge cycles (\cite{Ogus}*{Hope 4.11.3}). 
For any finite extension $L$ of $K$, an element in the tensor algebra of $\bigoplus_{i=0}^{2\dim X}H^i\dR(X/K)\otimes L$ is called an \emph{absolute Tate cycle} (\emph{loc. cit.} Def.~4.1) if it is fixed by all but finitely many crystalline Frobenii $\varphi_v$. 
When $v$ is unramified and $X$ has good reduction at $v$, the Frobenius $\varphi_v$ can be viewed as acting on $H^i_{\mathrm{dR}}(X/K)\otimes K_v$ via the canonical isomorphism between the de Rham and the crystalline cohomologies. Ogus proved that all the Hodge cycles are absolute Tate for abelian varieties and verified his prediction when $X$ is the product of abelian varieties with complex multiplication, Fermat hypersurfaces, and projective spaces (\emph{loc. cit.} Thm. 4.16). When $X$ is an elliptic curve, Ogus' conjecture follows from Serre--Tate theory.

\subsection{The main results}

Our first result is the following:

\begin{theorem}[\Cref{Qcase}]\label{thm_Q}
If $A$ is a polarized abelian variety over $\bQ$ and its $\ell$-adic algebraic monodromy group $G_\ell$ is connected, then the Mumford--Tate conjecture for $A$ implies that the absolute Tate cycles of $A$ coincide with its Hodge cycles. 
\end{theorem}

The Mumford--Tate conjecture for abelian varieties is known in many cases by the work of many people including Serre, Chi, Ribet, Tankeev, Pink, Banaszak, Gajda, Kraso\'n, and Vasiu. We refer the reader to \cite{Pink}, \cite{Vasiu}, \cite{BGK1}, \cite{BGK2} and their references. 

In \cref{int_pf}, we will explain the difficulty in generalizing \Cref{thm_Q} to an arbitrary number field $K$. However, we translate and generalize results in Serre's theory of Frobenius tori to our crystalline setting and then obtain a partial generalization of \Cref{thm_Q} to number fields by applying a theorem of Noot on the formal deformation space of an abelian variety with ordinary reduction. The following theorem presents two typical cases of our generalization and we refer the reader to \Cref{thm_K} for the complete statement.

\begin{theorem}\label{thm_main}
Let $A$ be a polarized abelian variety over $K$. If $A_{\bar{K}}$ is simple and either
\begin{enumerate}
\item $A$ is an abelian surface and $K$ is Galois over $\bQ$ of odd degree, or
\item the dimension of $A$ is a prime number and $\End_{\bar{K}}(A)$ is not $\bZ$,
\end{enumerate}
then the absolute Tate cycles of $A$ coincide with its Hodge cycles.
\end{theorem}

\subsection{A result without assuming the Mumford--Tate conjecture}
When the abelian variety $A$ over $K$ satisfies that $\End_{\bar{K}}(A)=\bZ$, Pink \cite{Pink} showed that there is a $\bQ$-model of $G_\ell^\circ$ which is independent of $\ell$ and ``looks like'' $G\MT$ in the following sense. The group $G\MT$ (resp. the $\bQ$-model of $G^\circ_\ell$) with its tautological faithful absolutely irreducible representation $H^1_{\mathrm{B}}(A,\bQ)$ (resp. $H^1_{\et}(A_{\bar{K}},\bQ_\ell)$) is an (absolutely) irreducible \emph{strong Mumford--Tate pair} over $\bQ$: the group is reductive and generated over $\bQ$ by the image of a cocharacter of weights $(0,1)$. Here weights mean the weights of the cocharacter composed with the faithful representation. Based on the work of Serre, Pink gave a classification of irreducible Mumford--Tate pairs. See \cite{Pink}*{Prop. 4.4, 4.5, and Table 4.6}.

In the crystalline setting, we define the \emph{absolute Tate group} $G_\aT$ of a polarized abelian variety $A$ over $\bQ$ to be the algebraic subgroup over $\bQ$ of $\GL(H^1_{\mathrm{dR}}(A/\bQ))$ stabilizing all of the absolute Tate cycles and we refer the reader to \cref{sub_GdR} for the definition of $G_\aT$ of $A$ over number field $K$. This group is reductive and we show that Pink's classification also applies to $G_\aT$ in the following situation:

\begin{theorem}[\Cref{MTpair}]\label{thm_QMT}
Assume that $A$ is a polarized abelian variety over $\bQ$ and the $\ell$-adic algebraic monodromy group of $A$ is connected. If $\End_{\bar{\bQ}}(A)=\bZ$, then the neutral connected component of $G_{\aT}$ with its tautological representation is an irreducible strong Mumford--Tate pair over $\bQ$.
\end{theorem}

It is worth noting that this theorem shows unconditionally that $G_{\aT}$ is of a very restricted form.

\subsection{The strategy of the proofs}\label{int_pf}
Some fundamental ingredients in the proofs of the Mumford--Tate conjecture are Faltings' isogeny theorem, $p$-adic Hodge theory (especially the weak admissibility), and Serre's theory of Frobenius tori. The later two still play important roles in our proofs and in analogy with Faltings' isogeny theorem, we use a theorem of Bost (\cite{B06}*{Thm. 6.4}) asserting that all the absolute Tate cycles lying in $\End(H^1_{\mathrm{dR}}(A/K)\otimes L)$ are algebraic and hence Hodge cycles. 

The key ingredient in the proof of \cite{B06}*{Thm.~6.4} is an algebraization theorem of Bost on foliations in commutative algebraic groups. We use this algebraization theorem to show that the absolute Tate group $G_\aT$ is reductive.

We first explain the main difficulties when $K$ is not equal to $\bQ$. For simplicity, we focus on the case that $\End_{\bar{K}}(A)=\bZ$. Pink's classification of Mumford--Tate pairs applies to connected reductive groups with an absolutely irreducible representation. 
The key step is to prove the irreducibility of $H^1\dR(A/K)$ as a $G^\circ_\aT$-representation. Notice that $G_\aT$ is not \emph{a priori} connected. 
In the $\ell$-adic setting, Serre, using the Chebotarev density theorem, showed that $G_\ell$ will be connected after passing to a finite extension.
There seems to be no easily available analogue argument for $G_\aT$. Another difficulty is that the centralizer of $G_\aT$ in $\End(H^1_{\mathrm{dR}}(A/K))$ does not \emph{a priori} coincide with the set of linear combinations of absolute Tate cycles in $\End(H^1_{\mathrm{dR}}(A/K))$. If it were the case, then Bost's theorem would imply that $H^1\dR(A/K)$ is at least an absolutely irreducible $G_\aT$-representation.

When $K=\bQ$ and $G_\ell$ is connected, the second difficulty disappears and the absolute Frobenii coincide with the relative ones. Thus the connectedness of $G_\ell$ implies that $G_\aT$ is \emph{almost connected}: $\varphi_p\in G^\circ_\aT(\bQ_p)$ for all $p$ inside a set of primes of natural density $1$. This fact implies that $H^1\dR(A/\bQ)\otimes \bar{\bQ}$ remains absolutely irreducible as a $G^\circ_\aT$-representation. This is a consequence of a strengthening of Bost's result (\Cref{Bost}): for any abelian variety $A$ over $K$, if $s\in \End(H^1\dR(A/K))$ satisfies the condition that $\varphi_v(s)=s$ for all finite places $v$ above a set of rational primes of natural density one, then $s$ is algebraic. We record a proof of this strengthening based on the work of Bost, Gasbarri and Herblot.

For a general $K$, the group $G_\ell$ only contains information about the relative Frobenii $\varphi_v^{m_v}$, where $m_v=[K_v:\bQ_p]$. Without additional inputs, it seems that the connectedness of $G_\ell$ would only imply that a variant of $G_\aT$ defined using $\varphi_v^{m_v}$ is almost connected. However, we do not have a variant of Bost's theorem with relative Frobenii beyond a couple of cases when the Mumford--Tate groups are small.

We now illustrate the idea of the proof of \Cref{thm_main}. The main task is to show that the centralizer of $G^\circ_\aT$ in $\End(H^1_{\mathrm{dR}}(A/K))$ coincides with that of $G\MT$. In case (2), since the Mumford--Tate group is not too large, we use Bost's theorem to show that otherwise $G_\aT^\circ$ must be a torus. Then we deduce that $A$ must have complex multiplication using a theorem of Noot (\cite{Noot}*{Thm. 2.8}) on formal deformation spaces at a point of ordinary reduction and hence we reduce this case to the case when $A$ has complex multiplication.
To exploit the strengthening of Bost's result to tackle case (1), we need to understand $\varphi_v$ for all finite places $v$ of $K$ lying over some set $M$ of rational primes of density $1$. While Serre's theorem on the ranks of Frobenius tori only provides information about completely split primes, we prove a refinement when $G_\ell=\GSp_{2g}$ that takes into account the other primes. Our result asserts that there exists a set $M$ of rational primes of density $1$ such that the Frobenius tori are of maximal rank for all $v$ lying over $M$. The rest of the argument is similar to that of case (2).

\subsection*{Organization of the paper}
\Cref{thm_Q} and \Cref{thm_QMT} are proved in \cref{MumfordTate} and \Cref{thm_main} is proved in \cref{connected}. 

As a first step, we prove in \cref{reductive} that, given an abelian variety $A$ over a number field $K$, its absolute Tate group $G_\aT$ is reductive. The key case is when $A$ is defined over $\bQ$ with $G_\ell$ connected. It is enough to show that the faithful representation $H^1\dR(A/\bQ)$ of $G_\aT$ is completely reducible. We use the algebraization theorem of Bost, Serre--Tate theory, and Grothendieck--Messing theory to deduce this from the semisimplicity of the isogeny category of abelian varieties.

In \cref{MumfordTate}, we recall the Mumford--Tate conjecture and the theory of Frobenius tori initiated by Serre. A result of Katz--Messing \cite{KM} enables us to view the Frobenius tori as subgroups of both $G_{\aT}$ and $G_\ell$. We use this observation and Serre's theorem on Frobenius tori to show that the Mumford--Tate conjecture implies that $G_\aT$ and $G\MT$ have the same rank. We then use this result and discussions in \cref{reductive} to prove \Cref{thm_Q} and \ref{thm_QMT}. We also prove two results that will be used in the proof of \Cref{thm_main}: our refinement of the results of Serre and Chi on the rank of Frobenius tori and an application of Noot's theorem on formal deformation.

In \cref{connected}, to prove \Cref{thm_main}, we first study the irreducible $G^\circ_\aT$-subrepresentations of $H^1\dR(A/K)\otimes \bar{K}$. This part is valid for all abelian varieties over number fields without assuming the Mumford--Tate conjecture. A main input is the result of Pink that $G_\ell$ with its tautological representation is a weak Mumford--Tate pair over $\bQ_\ell$. We then complete the proof of \Cref{thm_main} in \cref{sub_pf}.
 
In \cref{density}, we further generalize the result of Bost using the idea of Gasbarri and Herblot. The discussion contains a proof of \Cref{Bost}, the strengthening of Bost's theorem. We also discuss the consequence of our generalization on the absolute Tate cycles for abelian surfaces over quadratic fields. This is the first case when \Cref{thm_main} does not apply. 

\subsection*{Notation and conventions}
Let $K$ be a number field and $\cO_K$ its ring of integers. 
For a place $v$ of $K$, either archimedean or finite, let $K_v$ be the completion of $K$ with respect to $v$.
When $v$ is finite, we denote by $\fp$, $\cO_v$, and $k_v$ the corresponding prime ideal, the ring of integers, and residue field of $K_v$. We also denote by $p_v$ the characteristic of $k_v$ and when there is no confusion, we will also write $p$ for $p_v$. If there is no specific indication, $L$ denotes a finite extension of $K$.

For any vector space $V$, let $V^\vee$ be its dual and we use $V^{m,n}$ to denote $V^{\otimes m}\otimes (V^\vee)^{\otimes n}$. For a vector space $V$, we use $\GL(V), \GSp(V), \dots$ to denote the algebraic groups rather than the rational points of these algebraic groups.

For any scheme $X$ or vector bundle/space $V$ over $\Spec(R)$, we denote by $X_R'$ or $V_R'$ the base change to $\Spec R'$ for any $R$-algebra $R'$. For any archimedean place $\sigma$ of $K$ and any variety $X$ over $K$, we use $X_\sigma$ to denote the base change of $X$ to $\bC$ via a corresponding embedding $\sigma:K\rightarrow \bC$.

Throughout the paper, $A$ is a polarized abelian variety of dimension $g$ defined over $K$. 
For any field $F$ containing $K$, we denote by $H^i\dR(A/F)$ the de Rham cohomology group $$H^i\dR(A_F/F)=\bH^i(A_F,\Omega^{\bullet}_{A_F/F})=H^i\dR(A/K)\otimes_K F.$$
If $A$ has good reduction at $v$, we use $H^i\dR(A/k_v)$ to denote $H^i\dR(A_{k_v}/k_v)=H^i\dR(A_{\cO_v}/\cO_v)\otimes_{\cO_v}k_v$.

Let $v$ be a finite place of $K$.  If $A$ has good reduction at $v$ and $v$ is unramified in $K/\bQ$, we use $\varphi_v$ to denote the crystalline Frobenius acting on $H^1_{\mathrm{dR}}(A/K_v)^{m,n}$ via the canonical isomorphism to the crystalline cohomology $(H^1_{\mathrm{cris}} (A_{k_v}/W(k_v))\otimes K_v)^{m,n}$. 

The polarization of $A$ gives a natural identification between $H^1_{\mathrm{dR}}(A/K)$ and $(H^1_{\mathrm{dR}}(A/K))^\vee(-1)$, where $(-1)$ denotes the Tate twist, preserving filtration and Frobenius actions. Therefore, the Tate twist is implicitly considered when we say that an absolute Tate cycle is an element in $H^1_{\mathrm{dR}}(A/K_v)^{m,n}$ for some $m,n$.

We use $\dA$ to denote the dual abelian variety of $A$ and use $E(A)$ to denote the universal vector extension of $A$. We use $\End^\circ_?(A)$ to denote $\End_?(A)\otimes \bQ$, where $?$ can be $K,L$ or $\bar{K}$. The subscription is omitted if $\End_K(A)=\End_{\bar{K}}(A)$.

A reductive algebraic group in our paper may be nonconnected. Given an algebraic group $G$, we use $G^\circ$ to denote the neutral connected component and use $Z(G)$ to denote the center of $G$. When $G$ is reductive, the rank of $G$ means the rank of any maximal torus of $G$.

For any field $F$, we use $\bar{F}$ to denote a chosen algebraic closure of $F$. For any finite dimensional vector space $V$ over $F$ and any subset $S$ of $V$, we use $\Span_F(S)$ to denote the smallest sub $F$-vector space of $V$ containing $S$.

\subsection*{Acknowledgements}
I thank Mark~Kisin for introducing this problem to me and for the enlightening discussions and encouragement. I thank Olivier~Benoist, George~Boxer,Victoria Cantoral-Farf\'an, K\k{e}stutis~\v{C}esnavi\v{c}ius, Brian~Conrad, Peter~Jossen, Erick~Knight, Yifeng~Liu, Curtis~McMullen, Ananth~Shankar, Koji~Shimizu, Cheng-Chiang~Tsai, Shou-Wu~Zhang, Yihang~Zhu for helpful discussions or/and remarks on previous drafts of this paper. I thank the organizers and the participants of the 2016 JAVA conference. The interactions during the conference were both helpful for improving the presentation of the results in this paper and stimulating for future work. I thank the anonymous referees for their comments/suggestions on both the mathematics and the writing of this paper. 

\numberwithin{theorem}{subsection}
\section{Absolute Tate cycles and a result of Bost}\label{reductive}
In this section, we recall the definition of absolute Tate cycles and their basic properties in \cref{de Rham--Tate}. We define absolute Tate groups in \cref{sub_GdR} and this group will be our main tool to study Ogus' conjecture. We discuss Bost's theorem on absolute Tate cycles and its strengthening in \cref{sub_cent} and use Bost's algebraization theorem to prove that absolute Tate groups are reductive.

\subsection{Absolute Tate cycles}\label{de Rham--Tate}
\begin{defn}[\cite{Ogus}*{Def.~4.1}]\label{def_aTcycle}
An element $s\in (H^1_{\mathrm{dR}}(A/L))^{m,n}$ is called an \emph{absolute Tate cycle} of the abelian variety $A$ (over $L$) if there exists a finite set $\Sigma$ of finite places of $L$ such that  for all finite places $v\notin \Sigma$, $\varphi_v(s)=s$.
\end{defn}
\begin{rem}\label{rmk_def}
\begin{enumerate}
\item 
By \cite{Ogus}*{Cor. 4.8.1, 4.8.3}, an element $s\in (H^1_{\mathrm{dR}}(A/K))^{m,n}$ is absolute Tate if and only if its base change in $(H^1_{\mathrm{dR}}(A/L))^{m,n}$ is absolute Tate and the set of absolute Tate cycles over $L$ is stable under the natural action of $\Gal(L/K)$ (on the coefficient of de Rham cohomology groups). 
\item
By \cite{Ogus}*{Cor. 4.8.2}, although one could define absolute Tate cycles over arbitrary field $L$ containing $K$, we only need to consider cycles over number fields since any absolute Tate cycle must be defined over $\bar{\bQ}$ and hence over some number field.
\end{enumerate}
\end{rem}

Algebraic de Rham cohomology is equipped with a natural filtration $\Fil^\bullet$ (called \emph{the Hodge filtration}) from the $E_1$-page of Hodge-to-de Rham spectral sequence; see, for example, \cite{K72}*{1.4.1}. We have the following important fact.
\begin{lemma} [\cite{Ogus}*{Prop. 4.15}]\label{Fil0}
If $s\in H^1_{\mathrm{dR}}(A/L)^{m,n}$ is fixed by infinitely many $\varphi_v$ (for example, when $s$ is absolute Tate), then $s$ lies in $\Fil^0H^1_{\mathrm{dR}}(A/L)^{m,n}$. Moreover, if such $s$ lies in $\Fil^1H^1_{\mathrm{dR}}(A/L)^{m,n}$, then $s=0$.
\end{lemma}
\begin{proof}
By \cite{Mazur}*{Thm. 7.6} and the extension of the result to $H^1_{\mathrm{dR}}(A/L)^{m,n}$ in the proof by Ogus\footnote{The dual of $H^1_{\mathrm{cris}} (A_{k_v}/W(k_v))$ has a natural $W(k_v)$-structure, although $\varphi_v$ on the dual does not preserve this integral structure. In order to apply Mazur's argument to the dual, Ogus passes to a suitable Tate twist of the dual such that the new $\varphi_v$ acts integrally.}
, we have that for all but finitely many $v$, the mod $\fp$ filtration $\Fil^j(H^1_{\mathrm{cris}} (A_{k_v}/W(k_v)\otimes k_v)^{m,n})$ is the set $$\{\xi \text{ mod }\fp \mid \xi\in (H^1_{\mathrm{cris}} (A_{k_v}/W(k_v)))^{m,n} \text{ with } \varphi_v(\xi)\in p^j(H^1_{\mathrm{cris}} (A_{k_v}/W(k_v)))^{m,n}\}.$$
Then for the infinitely many $v$ such that $\varphi_v(s)=s$, we have that the reduction $$s \text{ mod }\fp\in \Fil^0((H^1_{\mathrm{cris}} (A_{k_v}/W(k_v))\otimes k_v)^{m,n})$$ and if $s\in \Fil^1H^1_{\mathrm{dR}}(A/L)^{m,n}$, then $s$ is $0$ modulo $\fp$. Since the Hodge filtration over $L$ is compatible with the Hodge filtration over $k_v$, we obtain the desired assertions.
\end{proof}

The main conjecture studied in this paper is the following:
\begin{conj}[Ogus \cite{Ogus}*{Problem 2.4, Hope 4.11.3}]\label{conj_main}
The set of absolute Tate cycles of an abelian variety $A$ defined over $K$ coincides with the set of Hodge cycles via the isomorphism between Betti and de Rham cohomologies.
\end{conj}

\begin{rem}\label{knowncase}
To study the conjecture, one may assume that $A$ is principally polarized. The reason is the following: after passing to some finite extension of $K$, the abelian variety $A$ is isogenous to a principally polarized one; moreover, this conjecture is insensitive to base change and the conjectures for two isogenous abelian varieties are equivalent. 
\end{rem}

\begin{theorem}[\cite{D82}*{Theorem 2.11}, \cite{Ogus}*{Theorem 4.14}, \cite{Bl94}]\label{HisdRT}
For any abelian variety,
every Hodge cycle is absolute Tate.
\end{theorem}

Therefore, the open part of \Cref{conj_main} is whether all of the absolute Tate cycles are Hodge cycles.

\subsection{The absolute Tate group}\label{sub_GdR}
We fix an isomorphism of $K$-vector spaces $H^1_{\mathrm{dR}}(A/K)$ and $K^{2g}$. Then the algebraic group $\GL_{2g,K}$ acts on $H^1_{\mathrm{dR}}(A/K)$ and hence on $H^1_{\mathrm{dR}}(A/L)^{m,n}$.

We first consider the case when $A$ is defined over $\bQ$. By \Cref{rmk_def}, set of absolute Tate cycles is stable under $\Gal(\bar{\bQ}/\bQ)$-action.
\begin{defn}[for $A$ over $\bQ$]
We define $G_\aT(A)$ to be the algebraic subgroup of $\GL_{2g,\bQ}$ such that for any $\bQ$-algebra $R$, the set of its $R$-points $G_\aT(A)(R)$ is the subgroup of $\GL_{2g}(R)$ which fixes all absolute Tate cycles. We call $G_\aT(A)$ the \emph{absolute Tate group} of the abelian variety $A$ over $\bQ$. We define $G_\aT^\bQ(A)$ to be the algebraic subgroup of $\GL_{2g,\bQ}$ which fixes all absolute Tate cycles over $\bQ$.
\end{defn}

\begin{rem}
Since $G_\aT(A)$ and $G_\aT^\bQ(A)$ are algebraic groups, they are defined as the stabilizers of finitely many absolute Tate cycles. Therefore, for all but finitely many $p$, we have $\varphi_p\in G_\aT^\bQ(\bQ_p)$.
\end{rem}

One could use the same definition to define the absolute Tate group for abelian varieties over number fields; however, it seems difficult to show directly that, for $A$ over a general number field, the stabilizer of all absolute Tate cycles is reductive. Therefore, we use the following definition. \Cref{conj_main} implies that these two definitions coincide. 
\begin{para}
We use $\Res^K_\bQ A$ to denote the abelian variety over $\bQ$ given by the Weil restriction. There is a natural map $A\rightarrow (\Res^K_\bQ A)_K$, which induces $H^1\dR(A/K)\hookrightarrow H^1\dR(\Res^K_\bQ A/K)$ as a direct summand. Since the projection from $H^1\dR(\Res^K_\bQ A/K)$ to $H^1\dR(A/K)$ is an absolute Tate cycle of $\Res^K_\bQ A$, the absolute Tate group $G_\aT(\Res^K_\bQ A)_K$ acts on $H^1\dR(A/K)$. 
\end{para}

\begin{defn}
We define $G_\aT(A)$ to be the image of the map $G_\aT(\Res^K_\bQ A)_K\rightarrow \GL(H^1\dR(A/K))$ and call $G_\aT(A)$ the \emph{absolute Tate group} of the abelian variety $A$ over $K$. When there is no risk of confusion, we will use $G_\aT$ to denote $G_\aT(A)$
\end{defn}

\begin{rem}\label{rmk_aTgroup}
By definition, $G_\aT(A)$ lies in the algebraic subgroup of $\GL(H^1\dR(A/K))$ which fixes all absolute Tate cycles of $A$. For all but finitely many finite places $v$, the injection $H^1\dR(A/K)\rightarrow H^1\dR(\Res^K_\bQ A/K)$ is $\varphi_v$-equivariant. Assume that $v|p$ is unramified in $K/\bQ$ and let $m_v=[K_v:\bQ_p]$. If $\varphi_v^{m_v}\in G_\aT(\Res^K_\bQ A)(K_v)$, then $\varphi_v^{m_v}\in G_\aT(A)(K_v)$. 
\end{rem}

\begin{lemma}\label{KdR}
Let $\{s_\alpha\}$ be a finite set of absolute Tate cycles of $\Res^K_\bQ A$ such that the algebraic group $G_\aT(\Res^K_\bQ A)$ is the stabilizer of all these $s_\alpha$. Let $K^{\aT}$ be the smallest finite extension of $K$ over which all these $s_\alpha$ are defined. Then
\begin{itemize}
\item the field $K^{\aT}$ is the smallest finite extension of $K$ over which all of the absolute Tate cycles of $\Res^K_\bQ A$ are defined;
\item the field $K^{\aT}$ is Galois over $K$.
\end{itemize}
\end{lemma}

\begin{proof}
Let $K^{\aT}$ be the smallest finite extension of $K$ over which all $s_\alpha$ in the finite set are defined. We need to show that if $t\in (H^1_{\mathrm{dR}}(\Res^K_\bQ A/\bar{\bQ}))^{m,n}$ is absolute Tate, then $t$ is defined over $K^{\aT}$. Let $L$ be a number field such that $t$ is defined and we may assume $L$ is Galois over $K^{\aT}$. 
Let $W$ be the sub vector space of $(H^1_{\mathrm{dR}}(\Res^K_\bQ A/L))^{m,n}$ spanned by $\{\gamma t \mid \gamma \in \Gal(L/K^{\aT})\}$. Since $W$ is $\Gal(L/K^{\aT})$ invariant, there exists a sub vector space $W_0$ of $(H^1_{\mathrm{dR}}(\Res^K_\bQ A/K^{\aT}))^{m,n}$ such that $W=W_0\otimes_{K^{\aT}} L$. By \Cref{rmk_def}, these $\gamma t$ are absolute Tate, and hence $G_\aT(\Res^K_\bQ A)(L)$ acts on $W$ trivially. In particular, $G_\aT(\Res^K_\bQ A)(K^{\aT})$ acts on $W_0$ trivially. On the other hand, since $\{s_\alpha\}\cup \{t\}$ is a finite set, for all but finitely many finite places $v$ of $L$, we have $\varphi_v(s_\alpha)=s_\alpha$ and $\varphi_v(t)=t$. Let $p$ be the residue characteristic of $v$ and let $m_v$ be $[K^{\aT}_v:\bQ_p]$. We use $\sigma_v$ to denote the Frobenius in $\Gal(K^{\rm{aT, nr}}_v/K^{\aT}_v)$. The $K_v^{\aT}$-linear action $\varphi_v^{m_v}$ lies in $G_\aT(\Res^K_\bQ A)(K^{\aT}_v)$ since it fixes all $s_\alpha$ and hence acts on $W_0\otimes K^{\aT}_v$ trivially. Let $\{w_i\}$ be a basis of $W_0$ and write $t=\sum_i a_iw_i$, where $a_i\in L$. Then $\varphi_v^{m_v}(t)=\sum\sigma_v(a_i)w_i$.
Since $\varphi_v^{m_v}(t)=t$, then $\sigma_v(a_i)=a_i$. Therefore, by the Chebotarev density theorem, $a_i\in K^{\aT}$ and $t\in W_0$. The last assertion of the lemma comes from \Cref{rmk_def}.
\end{proof}

\begin{lemma}\label{KtoQ}
The group $G_\aT(\Res^K_\bQ A)$ is a finite index subgroup of $G^\bQ_\aT(\Res^K_\bQ A)$.
\end{lemma}

\begin{proof}
Let $\{s_\alpha\}$ be a finite set of absolute Tate cycles of $\Res^K_\bQ A$ such that $G_\aT(\Res^K_\bQ A)$ is the stabilizer of all $s_\alpha$. Let $\{s_\alpha^1,\dots, s_\alpha^r\}\subset (H^1\dR(\Res^K_\bQ A/L))^{m,n}$ be the Galois orbit of $s_\alpha$ under the action of $\Gal(\bar{\bQ}/\bQ)$ and let $t^i_\alpha$ ($i\in\{1,\dots,r\}$) be the $i$-th symmetric power of $\{s_\alpha^1,\dots, s_\alpha^r\}$, that is, $t^1_\alpha=\sum_{i=1}^r s^i_\alpha\in (H^1\dR(\Res^K_\bQ A/L))^{m,n}$, $t^2_\alpha=\sum_{i\neq j}s^i_\alpha \otimes s^j_\alpha \in (H^1\dR(\Res^K_\bQ A/L))^{2m,2n}$,\dots. By \Cref{rmk_def}, these $t^i_\alpha$ are absolute Tate cycles over $\bQ$ and hence fixed by $G^\bQ_\aT(\Res^K_\bQ A)$. We consider the quotient (as a set, not with group structure) $G^\bQ_\aT(\Res^K_\bQ A)/G_\aT(\Res^K_\bQ A)$. 
The following claim is an elementary linear algebra result.
\begin{claim}
If $g\in \GL(H^1\dR(\Res^K_\bQ A/L))$ fixes $t^i_\alpha$ for all $i\in \{1,\dots, r\}$, then $g$ permutes the set $\{s_\alpha^1,\dots, s_\alpha^r\}$.
\end{claim}
\begin{proof}[Proof of Claim]
If not, assume that $g$ maps $\{s_\alpha^1,\dots, s_\alpha^r\}\subset (H^1\dR(\Res^K_\bQ A/L))^{m,n}$ to a different set of vectors $\{v^1,\dots, v^r\}$. We can always choose a projection of the $L$-vector space $(H^1\dR(\Res^K_\bQ A/L))^{m,n}$ to a line $W$ such that the projection of these two sets $\{\bar{s}_\alpha^1,\dots, \bar{s}_\alpha^r\}$ and  $\{\bar{v}^1,\dots, \bar{v}^r\}$ are different. Since $g$ fixes $t^i_\alpha$, it fixes the projection of $t^i_\alpha$ to $W^{\otimes i}$, which is the $i$-th symmetric power of $\{\bar{s}_\alpha^1,\dots, \bar{s}_\alpha^r\}$. In other words, $\{\bar{s}_\alpha^1,\dots, \bar{s}_\alpha^r\}$ and  $\{\bar{v}^1,\dots, \bar{v}^r\}$ have the same symmetric powers. Fix a basis vector $w$ of $W$ and write $\bar{s}^i_\alpha=a_iw$ and $\bar{v}^i=b_iw$ and let $f,g\in L[x]$ be the polynomial with roots $\{a_1,\dots, a_r\}$ and $\{b_1,\dots, b_r\}$. By assumption, $\{a_1,\dots, a_r\}$ and $\{b_1,\dots, b_r\}$ have the same symmetric powers, so $f=g$. Therefore $\{a_1,\dots, a_r\}=\{b_1,\dots, b_r\}$ and this contradicted with the assumption that $g$ does not permute the set $\{s_\alpha^1,\dots, s_\alpha^r\}$.
\end{proof}
By the above claim, every element in $G^\bQ_\aT(\Res^K_\bQ A)/G_\aT(\Res^K_\bQ A)$ preserves the Galois orbit of each $s_\alpha$ and hence this quotient set is finite.
\end{proof}

\begin{para}\label{GMT}\label{MTred}
Before we reformulate \Cref{conj_main} in terms of algebraic groups following the idea of Deligne, 
we recall the definition and basic properties of the Mumford--Tate group $G\MT$. See \cite{D82}*{Sec. 3} for details. When we discuss $G\MT$ and Hodge cycles, we always fix an embedding $\sigma:K\rightarrow \bC$. We denote $H^1_{\mathrm{B}}(A_\sigma(\bC),\bQ)$ by $V_B$. The vector space $V_B$ has a natural polarized Hodge structure of type $((1,0),(0,1))$. Let $\mu:\bG_{m,\bC}\rightarrow \GL(V_{B,\bC})$ be the \emph{Hodge cocharacter}, through which $z\in \bC^{\times}$ acts by multiplication with $z$ on $V_{B,\bC}^{1,0}$ and trivially on $V_{B,\bC}^{0,1}$.  The \emph{Mumford--Tate group} $G\MT$ of the abelian variety $A$ is the smallest algebraic subgroup defined over $\bQ$ of $\GL(V_B)$ such that its base change to $\bC$ containing the image of $\mu$. The Mumford--Tate group is the largest algebraic subgroup of $\GL(V_B)$ which fixes all Hodge cycles.\footnote{Here we consider Hodge cycles as elements in $V_B^{m,m-2i}(i)\subset V_B^{m',n'}$ for some choice of $m',n'$ as Tate twists is a direct summand of the tensor algebra of $V_B$.} Since all Hodge cycles are absolute Hodge cycles in the abelian variety case (\cite{D82}*{Thm. 2.11}), the algebraic group $G\MT$ is independent of the choice of $\sigma$. The Mumford--Tate group $G\MT$ is reductive (\cite{D82}*{Prop. 3.6}) and the fixed part of $G\MT$ in $V_B^{m,n}$ is the set of Hodge cycles. 
\end{para}

\begin{corollary}\label{dRTinMT}
Via the de Rham--Betti comparison, we have $G_{\aT,\bC}\subset G_{\rm{MT},\bC}$.
\end{corollary}
\begin{proof}
It follows from \Cref{HisdRT} and \Cref{rmk_aTgroup}.
\end{proof}

\Cref{conj_main} is essentially equivalent to the following conjecture, which we will mainly focus on from now on.

\begin{conj}\label{conj_gp}
Via the de Rham--Betti comparison, we have $G_{\aT,\bC}= G_{\rm{MT},\bC}$.
\end{conj}
\begin{proof}[Proof of equivalence] \Cref{conj_main} for $\Res^K_\bQ A$ implies \Cref{conj_gp} for $A$ by definition. Conversely, \Cref{conj_gp} for $A$ implies \Cref{conj_main} for $A$. Indeed, \Cref{conj_gp} for $A$ implies that every $\bC$-linear combination of absolute Tate cycles of $A$ is fixed by $G_{{\rm{MT}},\bC}$. Then by the discussion in \ref{MTred}, every $\bC$-linear combination of absolute Tate cycles maps to a $\bC$-linear combination of Hodge cycles via the de Rham--Betti comparison. We conclude by \Cref{HisdRT} and Prop.~4.9 in \cite{Ogus}, which shows that all absolute Tate cycles are $\bC$-linearly independent. 
\end{proof}

\begin{rem}
This conjecture implies that $G_{\aT}$ is connected and reductive. We will show that $G_\aT$ is reductive in \cref{sub_cent}. However, there seems to be no direct way to show that $G_\aT$ is connected without proving the above conjecture first.
\end{rem}

\subsection{Cycles in $\End(H^1\dR(A/K))$ and the reductivity of $G_\aT$}\label{sub_cent}

The following definition is motivated by \cite{H}*{Def. 3.5}.
\begin{defn}
An element $s$ of $(H^1_{\mathrm{dR}}(A/L))^{m,n}(i)$ is called a 
\emph{$\beta$-absolute-Tate cycle} if 
$$\beta\leq \liminf_{x\rightarrow \infty}\left(\sum_{v, p_v\leq x, \varphi_v(s)=s}\frac{[L_v:\bQ_{p_v}]\log p_v}{p_v-1}\right)\left([L:\bQ]\sum_{p\leq x}\frac{\log p}{p-1}\right)^{-1},$$
where $v$ (resp. $p$) runs over finite places of $L$ (resp. $\bQ$) and $p_v$ is the residue characteristic of $v$. 
\end{defn}

\begin{rem}
\leavevmode
\begin{enumerate}
\item
Absolute Tate cycles are $1$-absolute-Tate cycles by definition.
\item
Let $M$ be a set of rational primes with natural density $\beta$ and assume that for all $p\in M$ and for any $v|p$, one has $\varphi_v(s)=s$. Then $s$ is a $\beta$-absolute-Tate cycle by \cite{H}*{Lem. 3.7}.
\end{enumerate}
\end{rem}

\begin{theorem}[Bost, Gasbarri, Herblot]\label{Bost}
The set of $1$-absolute-Tate cycles in $\End(H^1_{\mathrm{dR}}(A/L))$ is the image of $\End_L(A)\otimes \bQ$. 
\end{theorem}
\begin{proof}
The statement restricted to absolute Tate cycles is a direct consequence of \cite{B01}*{Thm. 2.3} and we refer the reader to \cite{A04b}*{7.4.3} for a proof. See also \cite{B06}*{Thm. 6.4}. Notice that their argument is valid for $1$-absolute-Tate cycles if one applies the following proposition instead, which is a generalization of \cite{B01}*{Thm. 2.3}.
\end{proof}
\begin{proposition}\label{prop_density1}
Let $G$ be a commutative algebraic group over $L$ and let $W$ be an $L$-sub vector space of $\Lie G$. Assume that there exists a set $M$ of finite places of $L$ such that:
\begin{enumerate}
\item for any $v\in M$ over rational prime $p$, $W$ modulo $v$ is closed under $p$-th power map,
\item $\displaystyle \liminf_{x\rightarrow \infty}\left(\sum_{v, p_v\leq x, v\in M}\frac{[L_v:\bQ_{p_v}]\log p_v}{p_v-1}\right)\left([L:\bQ]\sum_{p\leq x}\frac{\log p}{p-1}\right)^{-1}=1.$
\end{enumerate} 
Then $W$ is the Lie algebra of some algebraic subgroup of $G$.
\end{proposition}
This proposition follows directly from \Cref{Gas}, which is a refinement of a theorem of Gasbarri by incorporating ideas of Herblot. We will prove this proposition in \cref{density} after \Cref{Gas}.

We now use \cite{B01}*{Thm.~2.3} to prove the reductivity of $G_\aT$. We start with $A$ over $\bQ$.
\begin{theorem}\label{red_aT_Q}
The absolute Tate group $G_\aT$ of an abelian variety $A$ over $\bQ$ is reductive.
\end{theorem}

\begin{proof}
The algebraic $\bQ$-group $G^\bQ_\aT$ admits a faithful representation $H^1\dR(A/\bQ)$. By the theory of reductive groups, $G_\aT^\bQ$ is reductive if this representation is completely reducible over $\bQ$.\footnote{To see this, let $R_u$ be the unipotent radical of $G_\aT^\bQ$. It is a normal unipotent subgroup of $G_\aT^\bQ$ over $\bQ$. Decompose $H^1\dR(A/\bQ)=\oplus V_i$, where $V_i$ are $\bQ$-irreducible representations of $G_\aT^\bQ$. Then the $R_u$-invariant part $V_i^{R_u}$ is a $G_\aT^\bQ$-representation over $\bQ$. By Lie--Kolchin theorem, $V_i^{R_u}$ is nonzero and hence $R_u$ acts trivially on $V_i$ and $H^1\dR(A/\bQ)$. This implies that $R_u=1$ and we conclude that $G_\aT^\bQ$ is reductive.} Since the isogeny category of abelian varieties is semisimple, we may assume that $A$ is simple over $\bQ$. Let $V\subset H^1\dR(A/\bQ)$ be a subrepresentation of $G^\bQ_\aT$. We want to show that $V$ is a direct summand. Indeed, we will show that $V=\{0\}$ or $H^1\dR(A/\bQ)$ when $A$ does not have complex multiplication (notice that $A$ is assumed to be simple over $\bQ$).\footnote{Even when $A$ has complex multiplication, one can still show that $V$ is trivial. However, since \Cref{conj_main} is known for $A$ with complex multiplication, we will not make extra effort to do so.} By duality, we may assume that $\dim V\geq \dim A$; otherwise, consider $(H^1\dR(A/\bQ)/V)^\vee(1) \subset (H^1\dR(A/\bQ)^\vee(1)\cong H^1\dR(A/\bQ)$. 

Since for all but finitely many primes $p$, the Frobenius $\varphi_p\in G^\bQ_\aT(\bQ_p)$, and then $\varphi_p$ acts on $V\otimes_{\bQ}\bQ_p$. Let $E(\dA)$ denote the universal vector extension of the dual abelian variety $\dA$. For any $p$ such that $A$ has good reduction at $p$, we have $\Lie E(\dA_{\bZ_p})=H^1\dR(A/\bZ_p)$ and by \cite{Mum}*{p.~138}, the action of $\varphi_p$ on $\Lie E(\dA_{\bF_p})$ is the $p$-th power map on derivations. Therefore, for all but finitely many $p$, the reduction of $V$ modulo $p$ is closed under $p$-th power map on derivations. By \cite{B01}*{Thm. 2.3}, $V$ is the Lie algebra of some $\bQ$-algebraic subgroup $H$ of $E(\dA)$. Since $A$ is simple over $\bQ$, the image of $H$ under the projection $E(\dA)\rightarrow \dA$ is either $\dA$ or $\{0\}$. 

We first consider the case when the image is $\dA$. Let $W$ be the vector group $(H\cap (\Lie A)^\vee)^\circ\subset E(\dA)$. We have the exact sequence
$$0\rightarrow (\Lie A)^\vee/W\rightarrow E(\dA)/W\rightarrow A\rightarrow 0.$$
The inclusion $H\subset E(\dA)$ induces $H/W\rightarrow E(\dA)/W$ and since $H/W$ is isogenous to $A$, this provides a splitting of the above exact sequence. Therefore, we have a surjection $E(\dA)/W\rightarrow (\Lie A)^\vee/W$. Since $E(\dA)$ is anti-affine (this follows from the universal property of $E(\dA)$), $(\Lie A)^\vee/W=\{0\}$ and hence $H=E(\dA)$. In other words, $V=H^1\dR(A/\bQ)$.

Now we deal with the case when $V=\Fil^1(H^1\dR(A/\bQ))$. In other words, for all but finitely many $p$, the crystalline Frobenius $\varphi_p$ preserves the filtration on $H^1\dR(A/\bQ_p)$. For such $p$, by Serre--Tate theory and Grothendieck--Messing theory, the Frobenius endomorphism of $A_{\bF_p}$ lifts to an endomorphism of $A_{\bZ_p}$. On the other hand, $\varphi_v(\Fil^1(H^1\dR(A/\bZ_p)))\subset p H^1\dR(A/\bZ_p)$ by Mazur's weak admissibility theorem. Therefore, the eigenvalues of $\varphi_p$ acting on $\Fil^1(H^1\dR(A/\bQ_p))$ are divisible by $p$. This implies that $A$ has ordinary reduction at $p$.
By \cite{Mes}*{Appendix Cor.~1.2}, $A$ is the canonical lifting of $A_{\bF_p}$ and hence has complex multiplication. In this case, \Cref{conj_main} is known and thus $G_\aT=G\MT$ is a torus. 

In summary, we have shown that $G^\bQ_\aT$ is reductive. By \Cref{KtoQ}, $G_\aT^\circ=(G_\aT^\bQ)^\circ$ and hence $G_\aT$ is also reductive. 
\end{proof}

\begin{corollary}\label{red_aT}
For any abelian variety $A$ over $K$, its absolute Tate group $G_\aT$ is reductive.
\end{corollary}
\begin{proof}
By definition, $G_\aT(A)$ is a quotient group of $G_\aT(\Res^K_\bQ A)$. By \Cref{red_aT_Q}, the group $G_\aT(\Res^K_\bQ A)$ is reductive and hence $G_\aT(A)$ is reductive.
\end{proof}

\section{Frobenius Tori and the Mumford--Tate conjecture}\label{MumfordTate}
In this section, we prove \Cref{thm_Q} and \Cref{thm_QMT}. 
One main input is the theory of Frobenius tori, which we recall in \cref{FrobT}. 
The fact that the Frobenius actions on the crystalline and \'etale cohomology groups have the same characteristic polynomial (\cite{KM}) enables us to view the Frobenius tori as subgroups of both $G_\aT$ and the $\ell$-adic algebraic monodromy group $G_\ell$. Hence the Frobenius tori serve as bridges between results for $G_\ell$ and those for $G_\aT$. In \cref{sub_MT}, we recall the Mumford--Tate conjecture and then use \Cref{Bost}, a key lemma of Zarhin and Serre's theorem to prove \Cref{thm_Q}. The extra inputs of the proof of \Cref{thm_QMT} are the weakly admissibility of certain filtered $\varphi$-modules of geometric origin and the Riemann hypothesis part of the Weil conjectures.
At the end of \cref{FrobT}, we prove a refinement (\Cref{TvQ}) of results of Serre and Chi on the rank of Frobenius tori. 
In \cref{sub_Noot}, we recall a result of Noot and use it to show that if $G^\circ_\aT$ of $A$ is a torus, then $A$ has complex multiplication.

From now on, we use $\Sigma$ to denote a finite set of finite places of $K^\aT$ containing all ramified places such that for $v\notin \Sigma$, the abelian variety $A_{K^{\aT}}$ has good reduction at $v$ and the Frobenius $\varphi_v$ stabilizes all of the absolute Tate cycles of $\Res^K_\bQ A$. For any finite extension $L$ of $K$ in question, we still use $\Sigma$ to denote the finite set of finite places $f^{-1}g(\Sigma)$, where $f:\Spec \cO_L\rightarrow \Spec \cO_K$ and $g:\Spec \cO_{K^{\aT}}\rightarrow \Spec \cO_K$. 

\subsection{Frobenius Tori}\label{FrobT}
The following definition is due to Serre. See also \cite{Chi}*{Sec. 3} and \cite{Pink}*{Sec. 3} for details.

\begin{defn} 
Assume that $L$ contains $K^\aT$. Let $T_v$ be the Zariski closure of the subgroup of $G_{\aT,L_v}$ generated by the $L_v$-linear map $\varphi_v^{m_v}\in G_\aT(L_v)$. Since $\varphi_v^{m_v}$ is semisimple (see for example \cite{Ogus}*{Lem.~2.10}), the group $T^\circ_v$ is a torus and is called the \emph{Frobenius torus} associated to $v$.
\end{defn}

\begin{rem}
The torus $T^\circ_v$ and its rank are independent of the choice of $L$. Moreover, $T_v^\circ$ is defined over $K_v$. 
\end{rem}

\begin{para}
For every prime $\ell$, we have the $\ell$-adic Galois representation 
$$\rho_\ell:\Gal(\bar{K}/K)\rightarrow \GL(H^1_{\et}(A_{\bar{K}},\bQ_\ell)),$$
and we denote by $G_\ell(A)$ the algebraic group over $\bQ_\ell$ which is the Zariski closure of the image of $\Gal(\bar{K}/K)$ and call $G_\ell(A)$ the \emph{$\ell$-adic algebraic monodromy group} of $A$. If it is clear which variety is concerned, we may just use $G_\ell$ to denote this group. Serre proved that there exists a smallest finite Galois extension $K^{\et}$ of $K$ such that for any $\ell$, the Zariski closure of the image of $\Gal(\overline{K^\et}/K^\et)$ is connected (\cite{Serre}*{Sec. 5, p. 15}).
\end{para}

\begin{rem}
For $v\nmid l$, we also view $T_v$ as an algebraic subgroup (only well-defined up to conjugation) of $G_\ell$ in the following sense. Since $A$ has good reduction at $v$, the action of the decomposition group at $v$ is unramified on $H^1_{\et}(A_{\bar{K}},\bQ_\ell)$. Since $v$ is unramified, we have an embedding $\Gal(\bar{k}_v/k_v)\cong \Gal(L^{nr}_v/L_v)\rightarrow \rho_\ell(\Gal(\bar{K}/K))$ after choosing an embedding $\bar{K}\rightarrow \bar{L}_v$. Hence we view the Frobenius $Frob_v$ as an element of $G_\ell$. By a result of Katz and Messing \cite{KM}, the characteristic polynomial of $\varphi_v^{m_v}$ acting on $H^1_{\mathrm{cris}}(A_{k_v}/W(k_v))$ is the same\footnote{To compare the two polynomials, we notice that both of them have $\bZ$-coefficients.} as the characteristic polynomial of $Frob_v$ acting on $H^1_\et(A_{\bar{K}},\bQ_\ell)$. Hence $T_v$ is isomorphic to the algebraic group generated by the semi-simple element $Frob_v$ in $G_\ell$. From now on, when we view $T_v$ as a subgroup of $G_\ell$, we identify $T_v$ with the group generated by $Frob_v$.
\end{rem}

Here are some important properties of Frobenius tori.

\begin{theorem}[Serre, see also \cite{Chi}*{Cor. 3.8}]\label{Tvmax}
There is a set $M_{max}$ of finite places of $K^\et$ of natural density one and disjoint from $\Sigma$ such that for any $v\in M_{max}$,  the algebraic group $T_v$ is connected and it is a maximal torus of $G_\ell$.
\end{theorem}

\begin{proposition}[\cite{Chi}*{Prop. 3.6 (b)}]\label{Tvconn_C}
For $L$ large enough (for instance, containing all the $n$-torsion points for some $n\geq 3$), all but finitely many $T_v$ are connected.
\end{proposition}

The following lemma is of its own interest and its proof is standard.

\begin{lemma}
Assume that $A$ is defined over $\bQ$. The number field $K^{\aT}$ is contained in $K^\et$. 
\end{lemma}

\begin{proof}
To simplify the notation, we enlarge $K^{\aT}$ to contain $K^\et$ and prove that they are equal.
Let $v$ be a finite place of $K^\et$ above $p$ such that $p$ splits completely in $K^\et/\bQ$ and identify $K^\et_v$ with $\bQ_p$ via $v$. Let $w$ be a place of $K^{\aT}$ above $v$. Denote by $\sigma$ the Frobenius in $\Gal(K^{\aT}_w/\bQ_p)=\Gal(K^{\aT}_w/K^\et_v)$. We consider the algebraic group $T_v$ generated by $\varphi_v\in G^{K^\et}_{\aT}(K^\et_v)$. If $v\in M_{max}$ as in \Cref{Tvmax}, then $T_v$ is connected and hence $T_v\subset G_{\aT,K^\et_v}$. This implies that $\varphi_v\in G_{\aT}(K^\et_v)$. For any $m,n$, let $W'\subset (H^1_{\mathrm{dR}}(A/K^{\aT}))^{m,n}$ be the $K^{\aT}$-linear span of all absolute Tate cycles in $(H^1_{\mathrm{dR}}(A/K^{\aT}))^{m,n}$. By \Cref{rmk_def}, there exists a $K$-linear subspace $W$ of $(H^1_{\mathrm{dR}}(A/K))^{m,n}$ such that $W'=W\otimes K^{\aT}$. Since $G_{\aT}$ acts trivially on $W'$ and $W$, the Frobenius $\varphi_v$ acts on $W\otimes_K K^\et_v$ trivially and $\phi_w$ acts on $W'\otimes_{K^{\aT}}K^{\aT}_w$ as the $\sigma$-linear extension of $\varphi_v$. Hence the elements in $W'$ that are stabilized by $\varphi_w$ are contained in $W\otimes_K K^\et_v$. That is to say that all absolute Tate cycles are defined over $K^\et_v$. As $m,n$ are arbitrary, we have $K^{\aT}_w=K^\et_v$. This implies that $p$ splits completely in $K^{\aT}/\bQ$ and hence $K^{\aT}= K^\et$ by the Chebotarev density theorem.
\end{proof}

\begin{rem}
For $A$ over $K$, the above lemma implies that $K^\aT$ is contained in $K^\et$ of $\Res^K_\bQ A$, which coincides with $K^\et$ of $A$ when $K/\bQ$ is Galois.
From \Cref{Bost}, we see that the field of definition of an absolute Tate cycle induced from an endomorphism of $A_{\bar{K}}$ is the same as the field of definition of this endomorphism. Hence $K^{\aT}$ contains the field of definition of all endomorphisms. Then, for $A$ over $\bQ$, the fields $K^{\aT}$ and $K^\et$ are the same if the field of definition of all endomorphisms is $K^\et$. This is the case when one can choose a set of $\ell$-adic Tate cycles all induced from endomorphisms of $A$ to cut out $G_\ell$. 
\end{rem}

Now we discuss some refinements of \Cref{Tvmax} and \Cref{Tvconn_C}. In the rest of this subsection, the field of definition $K$ of the polarized abelian variety $A$ is always assumed to be Galois over $\bQ$.
The main result is:
\begin{proposition} \label{TvQ}
Assume that $G^\circ_\ell(A)=\GSp_{2g}(\bQ_\ell)$. Then there exists a set $M$ of rational primes with natural density one such that for any $p\in M$ and any finite place $v$ of $K$ lying over $p$, the algebraic group $T_v$ generated by $\varphi_v^{m_v}$ (where $m_v=[K_v:\bQ_p])$ is of maximal rank. In particular, $T_v$ is connected for such $v$.\footnote{Our proof is a direct generalization of the proof of \Cref{Tvmax} by Serre. William Sawin pointed out to me that one may also prove this proposition by applying Chavdarov's method (\cite{Cha}) to $\Res^K_\bQ A$.}
\end{proposition}

The idea of the proof is to apply the Chebotarev density theorem to a suitably chosen Zariski closed subset of the $\ell$-adic algebraic monodromy group of $B$, the Weil restriction $\Res^K_\bQ A$ of $A$. As we are in characteristic zero, the scheme $B$ is an abelian variety over $\bQ$. We have $B_K=\prod_{\sigma\in \Gal(K/\bQ)}A^\sigma,$ where $A^\sigma=A\otimes_{K,\sigma}K$. It is a standard fact that $\dB=\Res^K_\bQ \dA$ and hence the polarization on $A$ induces a polarization on $B$ over $\bQ$. Moreover, the polarization on $A$ induces a polarization on $A^\sigma$. Extend $\sigma$ to a map $\sigma:\bar{K}\rightarrow \bar{K}$.
The map $\sigma:A(\bar{K})\rightarrow A^\sigma(\bar{K}), P\mapsto \sigma(P)$ induces a map on Tate modules $\sigma: T_\ell(A)\rightarrow T_\ell(A^\sigma)$.\footnote{Strictly speaking, there is a natural map $A^\sigma\rightarrow A$ obtained by the base change. We use $\sigma$ to denote its inverse on $\bar{K}$-valued points.} This map is an isomorphism between $\bZ_\ell$-modules. 

\begin{lemma}\label{sameGl}
The map $\sigma: T_\ell(A)\rightarrow T_\ell(A^\sigma)$ induces an isomorphism between the $\ell$-adic algebraic monodromy groups $G_\ell(A)$ and $G_\ell(A^\sigma)$.
\end{lemma}
\begin{proof}
Via $\sigma$, the image of $\Gal(\bar{K}/K)$ in $\End(T_\ell(A^\sigma))$ is identified as that of $\Gal(\bar{K}/K)$ in $\End(T_\ell(A))$. Hence $G_\ell(A)\simeq G_\ell(A^\sigma)$ as $T_\ell(-)^\vee=H^1_\et((-)_{\bar{K}},\bQ_\ell)$.
\end{proof}

We start from the following special case to illustrate the idea of the proof of \Cref{TvQ}. 
\begin{proposition}\label{TvQsp}
If $G_\ell(A)=\GSp_{2g,\bQ_\ell}$ and $A^\sigma$ is not geometrically isogenous to $A^\tau$ for any distinct $\sigma,\tau\in \Gal(K/\bQ)$, then there exists a set $M$ of rational primes with natural density $1$ such that for any $p\in M$ and any $v$ above $p$, the group $T_v$ is of maximal rank. That is, the rank of $T_v$ equals to the rank of $G_\ell(A)$. In particular, $T_v$ is connected for such $v$.
\end{proposition}
\begin{proof}
We use the same idea as in the proof of \Cref{Tvmax} by Serre. He constructed a proper Zariski closed subvariety $Z\subset G_\ell(A)$ with the following properties (see also \cite{Chi}*{Thm. 3.7}) and concluded by applying the Chebotarev density theorem to $Z$:
\begin{enumerate}
\item $Z$ is invariant under conjugation by $G_\ell(A)$, and
\item if $u\in G_\ell(A)(\bQ_\ell)\setminus Z(\bQ_\ell)$ is semisimple, then the algebraic subgroup of $G_\ell$ generated by $u$ is of maximal rank.
\end{enumerate}
Since $G_\ell(A)$ is connected, $Z(\bQ_\ell)$ is of measure zero in $G_\ell(A)(\bQ_\ell)$ with respect to the usual Haar measure.\footnote{Although Haar measures are only well-defined up to a constant, the property of being measure zero is independent of the choice and we fix one Haar measure and refer to it as the Haar measure.} We will define a Zariski closed subset $W\subset G_\ell(B)$ which has similar properties as $Z$.

Let $G^K_\ell(B)$ be the Zariski closure of $\Gal(\bar{K}/K)$ in $\GL(H^1_\et(B_{\bar{\bQ}},\bQ_\ell))$. Via the isomorphism $H^1_\et(B_{\bar{\bQ}},\bQ_\ell)\cong \oplus_{\sigma\in \Gal(K/\bQ)} H^1_\et(A_{\bar{K}}^\sigma,\bQ_\ell)$ of $\Gal(\bar{K}/K)$-modules, we view $G_\ell^K(B)$ as a subgroup of $\prod_{\sigma\in \Gal(K/\bQ)} G_\ell(A^\sigma)$. By the assumption that $A^\sigma$'s are not geometrically isogenous to each other and \cite{Lom}*{Thm. 4.1, Rem. 4.3}, we have $G^K_\ell(B)\cong\bG_m \cdot\prod_{\sigma\in \Gal(K/\bQ)} SG_\ell(A^\sigma)$, where $SG_\ell\subset G_\ell$ is the subgroup of elements with determinant $1$. Indeed, $\Lie SG_\ell(A^\sigma)=\fsp_{2g,\bQ_\ell}$ of type $C$ and the representations are all standard representations and then Rem.~4.3 in \emph{loc.~cit.} verified that Lombardo's theorem is applicable in our situation. Then By \Cref{sameGl}, we have $G^K_\ell(B)\simeq\bG_m\cdot SG_\ell(A)^{[K:\bQ]}$. This is the neutral connected component of $G_\ell(B)$.

The map $\Gal(\bar{\bQ}/\bQ)\rightarrow G_\ell(B)(\bQ_\ell)\rightarrow G_\ell(B)(\bQ_\ell)/G^K_\ell(B)(\bQ_\ell)$ induces a surjection $$\Gal(K/\bQ)\twoheadrightarrow G_\ell(B)(\bQ_\ell)/G^K_\ell(B)(\bQ_\ell).$$
Given $\sigma\in \Gal(K/\bQ)$, we denote by $\sigma G^K_\ell(B)$ the subvariety of $G_\ell(B)$ corresponding to the image of $\sigma$ in the above map. Let $m$ be the order of $\sigma\in \Gal(K/\bQ)$. We consider those $p$ unramified in $K/\bQ$ whose corresponding Frobenii in $\Gal(K/\bQ)$ fall into $c_\sigma$, the conjugacy class of $\sigma$. We have $m_v=m$ for all $v|p$.

Consider the composite map $m_\tau:\sigma G^K_\ell(B)\rightarrow G^K_\ell(B)\rightarrow G_\ell(A^\tau)\simeq G_\ell(A)$, where the first map is defined by $g\mapsto g^m$ and the second map is the natural projection. Let $W_{\sigma,\tau}$ be the preimage of $Z$ and $W_\sigma$ be $\cup_{\tau\in \Gal(K/\bQ)}W_{\sigma,\tau}$. By definition, $W_\sigma$ is a proper Zariski subvariety of the connected variety $\sigma G^K_\ell(B)$, and hence the measure of $W_\sigma(\bQ_\ell)$ is zero.
\begin{claim} If the Frobenius $Frob_p$ (well-defined up to conjugacy) is not contained in the conjugacy invariant set $\cup_{\gamma\in c_\sigma} W_\gamma(\bQ_\ell)$, then for any $v|p$, the algebraic subgroup $T_v\subset G_\ell(A)$ is of maximal rank. 
\end{claim}
\begin{proof} The subvariety $\cup_{\gamma\in c_\sigma} W_\gamma$ is invariant under the conjugation of $G^K_\ell(B)$ because $Z$ is invariant under the conjugation of $G_\ell(A)$. Moreover, this subvariety is invariant under the conjugation of $G_\ell(B)$ since $\tau W_\sigma\tau^{-1}=W_{\tau\sigma\tau^{-1}}$ by definition. By second property of $Z$ and the definition of the map $m_\tau$, we see that the image of $Frob_p^m$ generates a maximal torus in $G_\ell(A^\tau)$.
For each $v|p$, the Frobenius $Frob_v$ is the image of $Frob_p^m$ in $G_\ell(A^\tau)$ for some $\tau$ and hence $T_v$ is of maximal rank. 
\end{proof}
Let $W$ be $\cup_{\sigma\in \Gal(K/\bQ)}W_\sigma$. It is invariant under the conjugation of $G_\ell(B)$. As each $W_{\sigma,\tau}(\bQ_\ell)$ is of measure zero in $G_\ell(B)(\bQ_\ell)$, so is $W(\bQ_\ell)$. By the Chebotarev density theorem (see for example \cite{Se09}*{Sec. 6.2.1}), we conclude that there exists a set $M$ of rational primes with natural density $1$ such that $Frob_p\notin W(\bQ_\ell)$. Then the proposition follows from the above claim.
\end{proof}

\begin{rem}
The assumption that $G_\ell(A)=\GSp_{2g,\bQ_\ell}$ can be weakened. The proof still works if one has $G_\ell^K(B)=\bG_m\cdot \prod SG_\ell(A^\sigma)$. In other words, the proposition holds true whenever \cite{Lom}*{Thm. 4.1, Rem. 4.3} is applicable. For example, when $A$ has odd dimension and is not of type IV in Albert's classification.
\end{rem}

The following property of $\GSp_{2g}$ is used in an essential way of our proof of \Cref{TvQ}. It is well-known, but we give a proof for the sake of completeness.

\begin{lemma}\label{GSp}
If $G$ is an algebraic subgroup of $\GL_{2g, \bQ_\ell}$ containing $\GSp_{2g,\bQ_\ell}$ as a normal subgroup, then $G=\GSp_{2g,\bQ_\ell}$. In particular, $G^\circ_\ell(A)=\GSp_{2g,\bQ_\ell}$ implies that $G_\ell(A)$ is connected.
\end{lemma}
\begin{proof}
Let $g$ be a $\overline{\bQ_\ell}$-point of $G$. Then $\ad(g)$ induces an automorphism of $\GSp_{2g,\overline{\bQ_\ell}}$ by the assumption that $\GSp_{2g,\bQ_\ell}$ is a normal subgroup. As $\ad(g)$ preserves determinant, we view $\ad(g)$ as an automorphism of $\Sp_{2g,\overline{\bQ_\ell}}$. Since $\Sp_{2g,\overline{\bQ_\ell}}$ is a connected, simply connected linear algebraic group whose Dynkin diagram does not have any nontrivial automorphism, any automorphism of $\Sp_{2g,\overline{\bQ_\ell}}$ is inner. Hence $\ad(g)=\ad(h)$ for some $\overline{\bQ_\ell}$-point $h$ of $\Sp_{2g,\bQ_\ell}$. Then $g$ and $h$ differ by an element in the centralizer of $\Sp_{2g,\overline{\bQ_\ell}}$ in $\GL_{2g,\overline{\bQ_\ell}}$. Since the centralizer is $\bG_m$, we conclude that $g$ is in $\GSp_{2g,\bQ_\ell}(\overline{\bQ_\ell})$.
\end{proof}

\begin{proof}[Proof of \Cref{TvQ}]
Let $B$ be $\Res^K_\bQ A$. As in the proof of \Cref{TvQsp}, it suffices to construct a Zariski closed set $W\subset G_\ell(B)$ such that
\begin{enumerate}
\item $W(\bQ_\ell)$ is of measure zero with respect to the Haar measure on $G_\ell(B)(\bQ_\ell)$,
\item $W$ is invariant under conjugation by $G_\ell(B)$, and
\item if $u\in G_\ell(B)(\bQ_\ell)\setminus W(\bQ_\ell)$ is semisimple, then the algebraic subgroup of $G_\ell(B)$ generated by $u$ is of maximal rank.
\end{enumerate}

We first show that, to construct such $W$, it suffices to construct $W_\sigma\subset \sigma G^K_\ell(B)$ for each $\sigma\in \Gal(K/\bQ)$ such that 
\begin{enumerate}
\item $W_\sigma(\bQ_\ell)$ is of measure zero with respect to the Haar measure on $G_\ell(B)(\bQ_\ell)$,
\item $W_\sigma$ is invariant under conjugation by $G^K_\ell(B)$, and
\item if $u\in \sigma G_\ell^K(B)(\bQ_\ell)\setminus W(\bQ_\ell)$ is semisimple, then the algebraic subgroup of $G_\ell(B)$ generated by $u$ is of maximal rank.
\end{enumerate}
Indeed, given such $W_\sigma$, we define $W'$ to be $\cup_{\sigma\in \Gal(K/\bQ)}W_\sigma$. This set satisfies (1) and (3) and is invariant under conjugation by $G^K_\ell(B)$. We then define $W$ to be the $G_\ell(B)$-conjugation invariant set generated by $W'$. Since $[G_\ell(B):G^K_\ell(B)]$ is finite, $W$ as a set is a union of finite copies of $W'$ and hence satisfies (1) and (3). 

To construct $W_\sigma$, let $C\subset \Gal(K/\bQ)$ be the subgroup generated by $\sigma$. Consider $\{A^\tau\}_{\tau\in C}$. We have a partition $C=\sqcup_{1\leq i\leq r} C_i$ with respect to the $\bar{K}$-isogeny classes of $A^\tau$. These $C_i$ have the same cardinality $m/r$. For any $\alpha\in \Gal(K/\bQ)$, the partition of $\alpha C=\sqcup_{1\leq i\leq r}\alpha C_i$ gives the partition of $\{A^\tau\}_{\tau\in \alpha C}$ with respect to the $\bar{K}$-isogeny classes.

Consider the map $m_\alpha: \sigma G^K_\ell(B)\rightarrow G^K_\ell(B)\rightarrow G_\ell(A^\alpha)\simeq G_\ell(A)$ and define $W_{\sigma,\alpha}$ to be the preimage of $Z$ and $W_\sigma$ to be $\cup_{\alpha\in \Gal(K/\bQ)} W_{\sigma,\alpha}$ as in the proof of \Cref{TvQsp}. The proof of the claim there shows that $W_\sigma$ satisfies (2) and (3). 

Now we focus on (1). By the assumption and \Cref{GSp}, the group $G_\ell(A)$ is connected and hence $Z(\bQ_\ell)$ is of measure zero. Let $\gamma$ be $\sigma^r$. Consider $r:\sigma G^K_\ell(B)\rightarrow \gamma G^K_\ell(B)$ defined by $g\mapsto g^r$ and the composite map $(m/r)_\alpha: \gamma G^K_\ell(B)\rightarrow G^K_\ell(B)\rightarrow G_\ell(A^\alpha)\simeq G_\ell(A)$, where the first map is defined by $g\mapsto g^{m/r}$ and the second map natural projection. Then $m_\alpha=(m/r)_\alpha \circ r$. Let $W_r$ be $(m/r)_\alpha^{-1}(Z)$. Then $W_{\sigma,\alpha}=r^{-1}(W_r)$. Since any two of $\{A^\tau\}_{\tau=\alpha,\alpha\sigma,\cdots, \alpha\sigma^{r-1}}$ are not geometrically isogenous, the same argument as in the proof of \Cref{TvQsp} shows that if $W_r(\bQ_\ell)$ is of measure zero, so is $W_{\sigma,\alpha}(\bQ_\ell)$. The rest of the proof is to show that $W_r(\bQ_\ell)$ is of measure zero.

Notice that $G_\ell^\circ(B_K)=\bG_m\cdot\prod_{\sigma\in \bI}SG_\ell(A^\sigma)=\bG_m\cdot \Sp_{2g}^{|\bI|},$ where $\bI$ is a set of representatives of all isogeny classes in $\{A^\sigma\}_{\sigma\in \Gal(K/\bQ)}$.
 
Since the centralizer of $\GSp_{2g}$ in $\GL_{2g}$ is $\bG_m$ and $G^\circ_\ell(B)$ is a normal subgroup of $G_\ell(B)$, the map $(m/r)_\alpha$ is up to a constant the same as the following map:
$$\gamma G^K_\ell(B)\rightarrow \mathrm{Isom}_{\bQ_\ell}(H^1_\et(A^{\alpha\gamma}_{\bar{K}}, \bQ_\ell), H^1_\et(A^{\alpha}_{\bar{K}},\bQ_\ell))\cong \GL (H^1_\et(A^\alpha_{\bar{K}},\bQ_\ell))\rightarrow \GL(H^1_\et(A^\alpha_{\bar{K}},\bQ_\ell)),$$
where the first map is the natural projection, the middle isomorphism is given by a chosen isogeny between $A^\alpha$ and $A^{\alpha\gamma}$, and the last map is $g\mapsto g^{m/r}$.

The fact that $\gamma G^K_\ell(B)$ normalizes $G^\circ_\ell(B_K)$ allows us to apply \Cref{GSp} to the image of the above map and see that 
the above map factors through $$\GSp(H^1_\et(A^\alpha_{\bar{K}},\bQ_\ell))\rightarrow \GSp(H^1_\et(A^\alpha)), g\mapsto g^{m/r}$$ and hence $W_{\sigma,\alpha}(\bQ_\ell)$, being the preimage of a measure zero set under the above map, is of measure zero.
 \end{proof}

\subsection{The Mumford--Tate conjecture and the proofs of \Cref{thm_Q} and \Cref{thm_QMT}}\label{sub_MT}
\begin{conj}[Mumford--Tate]\label{conjMT}
For any rational prime $\ell$, we have $G_\ell^\circ(A)=G\MT(A)\otimes \bQ_\ell$ via the comparison isomorphism between the \'etale and the Betti cohomologies.
\end{conj}

\begin{lemma}\label{=rk}
If \Cref{conjMT} holds for the abelian variety $A$, then the reductive groups $G_\ell(A)$, $G_\aT(A)$, and $G\MT(A)$ have the same rank.
\end{lemma}

\begin{proof}
\Cref{conjMT} implies that $G_\ell$ and $G\MT$ have the same rank. Then by \Cref{Tvmax}, there are infinitely many finite places $v$ such that the Frobenius torus $T_v$ is a maximal torus of $G\MT$. Since, for these infinitely many places, $T_v$ is a subtorus of $G_\aT$ except for finitely many $v$, we have that $G_\aT$ has the same rank as $G\MT$ by \Cref{dRTinMT}. 
\end{proof}

The assertion of the above lemma is equivalent to the Mumford--Tate conjecture by the following lemma due to Zarhin and the Faltings isogeny theorem (see for example \cite{Vasiu}*{Sec. 1.1}).

\begin{lemma}[\cite{Zar}*{Sec. 5, key lemma}]\label{keylem}
Let $V$ be a vector space over a field of characteristic zero and $H\subset G\subset \GL(V)$ be connected reductive groups. Assume that $H$ and $G$ have the same rank and the same centralizer in $\End(V)$. Then $H=G$.
\end{lemma}

Using this lemma, we prove a special case of \Cref{conj_main}.

\begin{theorem}\label{Qcase}
If the polarized abelian variety $A$ is defined over $\bQ$ and $G_\ell(A)$ is connected, then the centralizer of $G_{\aT}^\circ$ in $\End(H^1_{\mathrm{dR}}(A/\bQ))$ is $\End(A)\otimes \bQ$ and moreover, \Cref{conjMT} implies \Cref{conj_main}.
\end{theorem}

\begin{proof}
The connectedness of $G_\ell$ implies that $K^\et=\bQ$ and hence $K^\aT=\bQ$.
By definition, for all $p\notin \Sigma$, the Frobenius $\varphi_p$ fixes all absolute Tate cycles and hence $\varphi_p\in G_\aT(\bQ_p)$. Thus $T_p(\bQ_p)\subset G_\aT(\bQ_p)$.
By \Cref{Tvmax}, we see that $T_p$ is connected for a density one set $M_{max}$ of rational primes $p$ disjoint from $\Sigma$.  Therefore, for all $p$ in $M_{max}$, the Frobenius $\varphi_p\in T_p(\bQ_p)\subset G^\circ_{\aT}(\bQ_p)$ and any $s$ in the centralizer of $G^\circ_\aT$ is fixed by $\varphi_p$. In other words, $s$ is a $1$-absolute-Tate cycle and by \Cref{Bost}, $s\in \End^\circ(A)$. The second assertion follows directly from \Cref{red_aT_Q}, \Cref{=rk}, and \Cref{keylem} (applying to $H=G_\aT^\circ$ and $G=G\MT$).
\end{proof}

As in \cite{Pink}, we show that if the conjecture does not hold, then $G_\aT$ is of a very restricted form when we assume that $A$ is defined over $\bQ$ and $K^\et=\bQ$. We need the following definition to state our result.

\begin{defn} [\cite{Pink}*{Def. 4.1}]
A \emph{strong Mumford--Tate pair} (of weight $\{0,1\}$) over $K$ is a pair $(G,\rho)$ of a reductive algebraic group over $K$ and a finite dimensional faithful algebraic representation of $G$ over $K$ such that there exists a cocharacter $\mu:\bG_{m,\bar{K}}\rightarrow G_{\bar{K}}$ satisfying:
\begin{enumerate}
\item the weights of $\rho\circ \mu$ are in $\{0,1\}$,
\item $G_{\bar{K}}$ is generated by $G(\bar{K})\rtimes \Gal(\bar{K}/K)$-conjugates of $\mu$.
\end{enumerate}
A pair $(G,\rho)$ is called a \emph{weak Mumford--Tate pair} if there exist finitely many cocharacters $\mu_1,\dots, \mu_d$ of $G_{\bar{K}}$ such that $\rho\circ \mu_i$ are all of weights $0$ and $1$ and that $G_{\bar{K}}$ is generated by $G(\bar{K})$-conjugates of $\mu_1,\dots, \mu_d$. 
\end{defn}

See \cite{Pink}*{Sec. 4, especially Table 4.6, Prop. 4.7} for the list of strong Mumford--Tate pairs.

\begin{theorem}\label{MTpair}
If the polarized abelian variety $A$ is defined over $\bQ$ and $G_\ell(A)$ is connected, then there exists a normal subgroup $G$ of $G_{\aT}^\circ$ defined over $\bQ$ such that
\begin{enumerate}
\item
$(G,\rho)$ is a strong Mumford--Tate pair over $\bQ$, where $\rho$ is the tautological representation $\rho: G\subset G_{\aT}\rightarrow \GL(H^1_{\mathrm{dR}}(A/\bQ))$, and
\item
The centralizer of $G$ in $\End(H^1_{\mathrm{dR}}(A/\bQ))$ is $\End^\circ(A)$.
\end{enumerate}
If we further assume that $\End(A)$ is commutative, then we can take $G$ to be $G^\circ_{\aT}$.
\end{theorem}

The following lemma constructs the cocharacter $\mu$. 
\begin{lemma}\label{Kbarcochar}
There exists a cocharacter $\mu:\bG_{m,\bar{K}}\rightarrow G_{\aT,\bar{K}}^\circ$ such that its induced filtration on $H^1\dR(A/\bar{K})$ is the Hodge filtration. Moreover, different choices of such cocharacters are conjugate by an element of $G_{\aT,\bar{K}}^\circ(\bar{K})$.
\end{lemma}
\begin{proof}
We use $v\notin \Sigma$ to denote a finite place of $K^\et$ (of $\Res^K_\bQ A$) such that $v$ splits completely in $K^\et/\bQ$. By \Cref{Tvmax}, there is a density one set (over $K^\et$) of places $v$ such that $T_v$ is connected and hence $\varphi_v\in G^\circ_\aT(K_v)$. Since $G^\circ_\aT$ is reductive (by \Cref{red_aT}), by \cite{D82}*{Prop.~3.1 (c)}, $G^\circ_\aT$ is the stabilizer of a finite set of elements $\{t_\alpha\}$ in $(H^1\dR(A/K))^{m_\alpha,n_\alpha}$. Since $\varphi_v\in G^\circ_\aT(K_v)$, we have $\varphi_v(t_\alpha)=t_\alpha$. Then by \Cref{Fil0}, $t_\alpha\in \Fil^0(H^1\dR(A/K))^{m_\alpha,n_\alpha}$. Since $H^1\dR(A/K_v)$ is a weakly admissible filtered module, by \cite{Kisin}*{Lem.~1.1.1 and Lem.~1.4.5}, there is a cocharacter $\mu_v: \bG_{m,K_v}\rightarrow G_{\aT,K_v}^\circ$ whose induced filtration is the Hodge filtration.

By \cite{Kisin}*{Lem.~1.1.1} and the existence of $\mu_v$, the subgroup $P$ of $G_{\aT,\bar{K}}^\circ$ preserving the Hodge filtration is parabolic and the subgroup $U$ of $G_{\aT,\bar{K}}^\circ$ acting trivially on the graded pieces of $H^1\dR(A/\bar{K})$ is the unipotent radical of $P$.\footnote{Strictly speaking, the existence of $\mu_v$ shows that after base change to $K_v$, the group $P_{K_v}$ is parabolic and $U_{K_v}$ is its unipotent radical. However, these two properties still hold over any subfield of $K_v$.} Moreover, the action of $P$ on the graded pieces is induced by a cocharacter of $P/U$. Then given a Levi subgroup of $P$, one can construct a cocharacter of $G_{\aT,\bar{K}}^\circ$ inducing the desired filtration and vice versa. Therefore, the assertions follow from the existence of a Levi subgroup over $\bar{K}$ and the fact that any two of such Levi subgroups are conjugate.
\end{proof}

\begin{proof}[Proof of \Cref{MTpair}]
Let $\mu$ be some cocharacter constructed in \Cref{Kbarcochar} and $G$ be the smallest normal $\bQ$-subgroup of $G_{\aT}^\circ$ such that $G(\bar{\bQ})$ contains the image of $\mu$. We first show that $(G,\rho)$ is a strong Mumford--Tate pair over $\bQ$.\footnote{The pair $(G,\rho)$ is canonical. Indeed, different choices of $\mu$ are conjugate to each other over $\bar{\bQ}$ by an element in $G_\aT^\circ(\bQ)$ and hence the definition of $G$ is independent of the choice of $\mu$.}

The weights of $\rho\circ\mu$ are $0$ or $1$ since the non-zero graded pieces of the Hodge filtration on $H^1\dR(A/\bar{K})$ are at $0$ and $1$. Since the subgroup of $G$ generated by $G_{\aT}^\circ(\bar{\bQ})\rtimes \Gal(\bar{\bQ}/\bQ)$-conjugates of $\mu$ must be defined over $\bQ$ and normal in $G^\circ_\aT$, this subgroup coincides with $G$. Since $G_{\aT}^\circ$ is connected and reductive (by \Cref{red_aT}) and the image of $\mu$ is contained in $G$, the set of $G_{\aT}^\circ(\bar{\bQ})$-conjugates of $\mu$ is the same as the set of $G(\bar{\bQ})$-conjugates of $\mu$. Hence $(G,\rho)$ is a strong Mumford--Tate pair over $\bQ$.\\

To show (2), by \Cref{Bost}, it suffices to show that $\varphi_p\in G(\bQ_p)$ for $p$ in a set of natural density $1$. By \Cref{Tvmax}, it suffices to show that for any $p\in M_{max}$, there exists an integer $n_p$ such that $\varphi_p^{n_p}\in G(\bQ_p)$.\footnote{By standard arguments, one can choose an $n$ independent of $p$.} To see this, we will use the weak admissibility of crystalline cohomology groups and the Riemann hypothesis part of the Weil conjecture for varieties over finite fields.

 For any $m,n$, let $W\subset (H^1_{\mathrm{dR}}(A/\bQ))^{m,n}$ be the largest $\bQ$-sub vector space with trivial $G$-action. Since $G$ is normal in $G^\circ_{\aT}$, the group $G^\circ_\aT$ acts on $W$.
Then for all $p\in M_{max}$, we have $\varphi_p\in G^\circ_\aT(\bQ_p)$ acts on $W\otimes \bQ_p$.
Since $G$ is reductive, there exist a finite set of such $W$ such that $G$ is the largest subgroup of $\GL(H^1\dR(A/\bQ))$ acting trivially on these $W$.
Since $\varphi_v\in G^\circ_\aT(\bQ_p)$ is semi-simple, in order to show that $\varphi_p^{n_p}\in G(\bQ_p)$, it suffices to prove that the eigenvalues of $\varphi_p$ acting on $W\otimes\bQ_p$ are all roots of unity.

Since $W\subset (H^1_{\mathrm{dR}}(A/\bQ))^{m,n}$, the eigenvalues of $\varphi_p$ are all algebraic numbers. Since $Frob_p$ acts on $(H^1_{\et}(A_{\bar{\bQ}},\bQ_\ell))^{m,n}(i)$ with all eigenvalues being $\ell$-adic units for $\ell\neq p$, the eigenvalues of $\varphi_p$ are also $\ell$-adic units. 
Now we show that these eigenvalues are $p$-adic units.
Let $H_p$ be the Tannakian fundamental group of the abelian tensor category generated by sub weakly admissible filtered $\varphi$-modules of $(H^1\cris(A/W(\bF_p))\otimes \bQ_p)^{m,n}$. For $p\in M_{max}$, by \cite{Pink}*{Prop. 3.13}, $H_p$ is connected. By \Cref{Fil0}, every absolute Tate cycle generates a trivial filtered $\varphi$-module. Then by the definition of $H_p$, we have $H_p(\bQ_p)\subset G_\aT(\bQ_p)$ and thus $H_p(\bQ_p)\subset G^\circ_\aT(\bQ_p)$. Hence $W\otimes \bQ_p$ is an $H_p$-representation and then by the Tannakian equivalence, the filtered $\varphi$-module $W\otimes \bQ_p$ is weakly admissible. 
Since $\mu$ is a cocharacter of $G$, then $\mu$ acts on $W\otimes \bQ_p$ trivially and hence by \Cref{Kbarcochar}, the filtration on $W\otimes \bQ_p$ is trivial. Then the Newton cocharacter is also trivial. In other words, the eigenvalues of $\varphi_p$ are $p$-adic units. By the Weil conjecture, the archimedean norms of the eigenvalues are $p^{(m-n)/2}$. Then by the product formula, the weight $\frac{m-n}2$ must be zero and all the eigenvalues are roots of unity.\\

We now prove the last assertion.
If $G\neq G_{\aT}^\circ$, then we have $G_{\aT}^\circ=GH$ where $H$ is some nontrivial normal connected subgroup of $G_{\aT}^\circ$ commuting with $G$ and $H\cap G$ is finite.\footnote{To see this, notice that as a connected reductive group, $G_{\aT}^\circ$ is the product of central torus and its derived subgroup and the intersection of these two subgroups is finite. Then one uses the decomposition results for tori and semi-simple groups.} Then $H$ is contained in the centralizer of $G$ and by (2), we have $H\subset \End^\circ(A)$. By the assumption on $\End^\circ(A)$, we see that $H$ is commutative and hence $H\subset Z^\circ(G^\circ_\aT)$, the neutral connected component of $Z(G_\aT^\circ)$. We will draw a contradiction by showing that $Z^\circ(G^\circ_\aT)\subset G$. By \Cref{Qcase}, we have 
$$Z(G^\circ_\aT)=G^\circ_\aT\cap\End^\circ(A)\subset G\MT \cap \End^\circ(A)=Z(G\MT),$$ and hence $Z^\circ(G^\circ_\aT)\subset Z^\circ(G\MT)$.
On the other hand, for all $p\in M_{max}$, the torus $T_p\subset G$. Hence we only need to show that $Z^\circ(G\MT)\subset T_p$ for some $p\in M_{max}$. Since it is enough to prove the statement on geometric points and the Frobenius torus generated by $\varphi_p$ maps to the torus generated by $Frob_p$ via an automorphism of $G\MT$ by \cite{K17}*{Cor.~2.3.1}, we only need to show that $Z^\circ(G\MT\otimes \bQ_\ell)\subset T_p\subset \GL(H^1_\et(A_{\bar{\bQ}},\bQ_\ell))$, where $T_p$ now denotes the torus generated by $Frob_p$. Since $T_p$ is a maximal torus, we have $T_p\supset Z^\circ(G_\ell)$. We then conclude by \cite{Vasiu}*{Thm.~1.3.1} asserting that $Z^\circ(G\MT\otimes \bQ_\ell)=Z^\circ(G_\ell)$.
\end{proof}

\subsection{A result of Noot and its consequence}\label{sub_Noot}
It is well-known that the Mumford--Tate group of $A$ is a torus if and only if $A$ has complex multiplication. In particular, $G_{\aT}^\circ(A)$ is a torus when $A$ has complex multiplication. In this subsection, we will show that the converse is also true.

\begin{proposition}\label{cm=1pt}
If $G_\aT^\circ$ commutes with the cocharacter $\mu$ constructed in \Cref{Kbarcochar},\footnote{This assumption does not depend on the choice of $\mu$, since different choices of $\mu$ are the same up to conjugation by elements in $G^\circ_\aT(\bar{K})$.} then $A$ has complex multiplication and hence \Cref{conj_main} holds for $A$.
\end{proposition}

\begin{corollary}\label{cm=torus}
The abelian variety $A$ has complex multiplication if and only if $G_\aT^\circ$ is a torus.
\end{corollary}

One may prove \Cref{cm=1pt} by using the idea in the second last paragraph in the proof of \Cref{red_aT}. We provides a proof here using a theorem of Noot on the deformation space of an ordinary abelian variety in positive characteristic. The following lemma shows that we have enough places with ordinary reduction to apply Noot's theorem.

\begin{lemma}\label{ordinary}
If $G_\aT^\circ$ commutes with the cocharacter $\mu$ constructed in \Cref{Kbarcochar}, then $A$ has ordinary reduction at a positive density of primes of degree one (that is, splitting completely over $\bQ$).
\end{lemma}

\begin{proof}
This is a direct consequence of the weak admissibility.
After replacing $K$ by a finite extension, we may assume that some fixed cocharacter $\mu$ constructed in \cref{Kbarcochar} is defined over $K$.
Let $v$ be a finite place of $K$ with residue characteristic $p$ and assume that $p$ splits completely in $K/\bQ$. Then $K_v\cong\bQ_p$. 
Let $\nu_v$ be the Newton (quasi-)cocharacter and fix a maximal torus $T\subset G_\aT^\circ$. Since the Frobenius torus $T_v^\circ$ is contained in $G^\circ_\aT$, the Newton cocharacter  $\nu_v$ factors through $G_\aT^\circ$. As in \cite{Pink}*{Sec.~1}, we define $S_{\mu}$ (resp. $S_{\nu_v}$) to be the set of $G_\aT^\circ(\overline{K}_v)\rtimes \Gal(\overline{K}_v/K_v)$-conjugates of $\mu$ (resp. $\nu_v$) factoring through $T(\overline{K}_v)$ in $\Hom((\bG_{m})_{\overline{K}_v}, T_{\overline{K}_v})\otimes_{\bZ}\bR$.
Since all the $G_\aT^\circ(\overline{K}_v)$-conjugacy of $\mu$ coincide with itself and $\mu$ is defined over $K_v$, we have $S_{\mu}=\{\mu\}$. 
By the weak admissibility, we have that $S_{\nu_v}$ is contained in the convex polygon generated by $S_{\mu}$ (see \cite{Pink}*{Thm.~1.3, Thm.~2.3}). Hence $S_{\nu_v}=S_{\mu}$.
Then we conclude that $A$ has ordinary reduction at $v$ by \cite{Pink}*{Thm.~1.5} if $v\notin \Sigma$.
\end{proof}

\begin{para}
We reformulate our assumption in the language of Noot. As in the proof of \Cref{Kbarcochar}, there is a finite set $\{t_\alpha\}$ with $t_\alpha\in (H^1_{\mathrm{dR}}(A/K))^{m_\alpha,n_\alpha}$ such that $G_\aT^\circ$ is the stabilizer of $\{t_\alpha\}$. We have shown that there is a density one set of places $v$ (over some finite extension of $K$) such that $\varphi_v(t_\alpha)=t_\alpha$ and $t_\alpha \in \Fil^0(H^1_{\mathrm{dR}}(A/K))^{m_\alpha,n_\alpha}$. We assume that $v\notin \Sigma$ and $A$ has ordinary good reduction at $v$. The later assumption holds for infinitely many $v$ under the assumption of \Cref{cm=1pt} by \Cref{ordinary}.
Such $t_\alpha$ is a `Tate cycle' in the sense of \cite{Noot}. In the formal deformation space of $A_{k_v}$, Noot defined the formal locus $\cN$ where the horizontal extensions of all $t_\alpha$ are still in $\Fil^0(H^1_{\mathrm{dR}}(A/K))^{m_\alpha,n_\alpha}$ (see \cite{Noot}*{Sec. 2} for details).
\end{para}

\begin{theorem}[\cite{Noot}*{Thm. 2.8}\footnote{Ananth~Shankar pointed out to me that one may use Kisin's results in \cite{Kisin}*{Sec. 1.5} and the property of canonical lifting to prove this result.}]\label{Noot}
The formal locus $\cN$ is a translate of a formal torus by a torsion point. Moreover, the dimension of $\cN$ equals to the dimension of the unipotent radical of the parabolic subgroup of $G_\aT^\circ$ which preserves the Hodge filtration of $A$.
\end{theorem}
\begin{proof}
We use $\Gr(\dim A, H^1\dR(A/\overline{K}_v))$ to denote the Grassmannian parametrizing vector space $V\subset H^1\dR(A/\overline{K}_v)$ with $\dim V=\dim A$.
\cite{Noot}*{Thm.~2.8} asserts that $\dim \cN$ equals the dimension of the subvariety of $\Gr(\dim A, H^1\dR(A/\overline{K}_v))$ such that $t_\alpha$ remains in $\Fil^0$ for the filtration defined by $V$. In order to compute the dimension, we may restrict ourselves to $V$ over $K_v^{\rm{nr}}$ such that $V$ modulo $p_v$ equals the Hodge filtration on $H^1\dR(A/\overline{K}_v)$. By \cite{Kisin}*{Lem.~1.4.5}, such $V$ with $t_\alpha\in \Fil^0$ gives raise to a cocharacter $\mu_V$ into $G^\circ_\aT$. Since $\dim V=\dim A$, this cocharacter $\mu_V$ is conjugate to $\mu$ in \Cref{Kbarcochar} by elements in $\GL(H^1\dR(A/\overline{K}_v))$. Therefore, the set of $\mu_V$ which is conjugate to $\mu$ by elements in $G^\circ_\aT$ is of finite index in the set of all possible $\mu_V$. In other words, $\dim \cN$ equals to the dimension of the variety parametrizing $V$ such that $\mu_V$ is conjugate to $\mu$ by an element in $G^\circ_\aT$. This variety is $G^\circ_\aT/P$, where $P$ is the parabolic subgroup of $G^\circ_\aT$ that preserves the Hodge filtration. We conclude by the fact that $\dim G^\circ_\aT/P=\dim U$, where $U$ is the unipotent radical of $P$.
\end{proof}

\begin{proof}[Proof of \Cref{cm=1pt}]
By \Cref{Noot} and the assumption that $G_\aT^\circ$ commutes with $\mu$, the dimension of $\cN$ is $0$ and thus $A_{K_v}$ is a torsion point in the formal deformation space. By Serre--Tate theory, a torsion point corresponds to an abelian variety with complex multiplication. Hence $A$ has complex multiplication and the last assertion comes from \Cref{knowncase}.
\end{proof}

\section{The main theorem over $K$ and its proof}\label{connected}
In this section, we extend \Cref{thm_Q} to abelian varieties over some number field $K$. The main result is \Cref{thm_K}, which contains \Cref{thm_main} as special cases. The main difficulty of the proof is to show that $G\MT$ and $G_\aT^\circ$ have the same centralizer in $\End(H^1\dR(A/K))$. 
In \cref{sub_gpth}, we first show that in order to show the agreement of centralizers, one only needs to show that all irreducible subrepresentations of $G\MT$ in $H^1\dR(A/K)\otimes \bar{K}$ remain irreducible as representations of $G_\aT^\circ$. We also show that for $A$ simple, all irreducible subrepresentations of $G_\aT^\circ$ in $H^1\dR(A/K)\otimes \bar{K}$ are of same dimension and $\varphi_v$ permutes these irreducible subrepresentations. Two key inputs are a result of Pink classifying the tautological representation of $G_\ell^\circ$ and a result of Kisin (see \cref{dRtoet}) relating \'etale and crystalline Frobenii. 
In \cref{sub_pf}, we use the results in \cref{sub_gpth}, \Cref{Bost}, \Cref{TvQ} and \Cref{cm=1pt} to show that under the assumptions of \Cref{thm_K}, the centralizer of $G^\circ_\aT$ in $\End(H^1\dR(A/\bar{K}))$ coincides with that of $G\MT$ and then complete the proof.

\subsection{Group theoretical discussions}\label{sub_gpth}
\begin{para}\label{assumpA}
Throughout this section, $v$ is a finite place such that the algebraic group $G\MT(A)$ over $\bQ$ is unramified at $p_v$, the residue characteristic of $v$. This holds for all but finitely many $v$, so restricting ourselves to such $v$ is harmless for later applications. This condition is the assumption in \cite{K17}, which will be used in \cref{dRtoet}.
Let $\rho\MT: G\MT\rightarrow \GL(H^1_{\mathrm{dR}}(A/K))$ be the tautological algebraic representation of the Mumford--Tate group on de Rham cohomology.\footnote{All Hodge cycles are absolute Tate and hence defined in $(H^1\dR(A/\bar{K}))^{m,n}$. Moreover, since all Hodge cycles are absolute Hodge, the action of $\Gal(\bar{K}/K)$ on the coefficients of de Rham cohomology preserves the set of Hodge cycles. Therefore, the Mumford--Tate group, as the stabilizer of all Hodge cycles, is a $K$-subgroup of $\GL(H^1_{\mathrm{dR}}(A/K))$.} We denote by $\rho_{\bar{K}}:G_{{\rm{MT}},\bar{K}}\rightarrow \GL(H^1_{\mathrm{dR}}(A/\bar{K}))$ the representation obtained by base change to $\bar{K}$. Assume that the $G_{{\mathrm{MT}},\bar{K}}$-representation $\rho_{\bar{K}}$ decomposes as $\rho_{\bar{K}}=\bigoplus_{i=1}^n\rho_{\bar{K},i}$ and each component decomposes as $\rho_{\bar{K},i}|_{G^\circ_{\aT,\bar{K}}}=\bigoplus_{j=1}^{n_i}\rho_{\bar{K},i,j}$, where $\rho_{\bar{K},i}$'s (resp. $\rho_{\bar{K},i,j}$'s ) are irreducible representations of $G_{{\mathrm{MT}},\bar{K}}$ (resp. $G_{\aT,\bar{K}}^\circ$). We denote the vector space of $\rho_{\bar{K},i}$ (resp. $\rho_{\bar{K},i,j}$) by $V_i$ (resp. $V_{i,j}$).
\end{para}

\begin{para}\label{dRtoet}
We now explain how to use \cite{K17}*{Cor.~2.3.1} to translate problems on representations of $G^\circ_\aT$ to problems on representations of $G_\ell$.\footnote{In a previous version of this paper, we used a result of Noot \cite{N09}*{Thm.~4.2}, which assumes that $(G\MT(A)_{\bar{\bQ}})^{\der}$ does not have any simple factor of type $\SO_{2k}$ for $k\geq 4$. We thank one of the referees for pointing out to me this result of Kisin, which shows the same result without hypothesis on the type of $G\MT(A)$.} This is the main idea of the proofs of \Cref{nonisom} and \Cref{nomoresub}. We fix an embedding $\bar{K}\rightarrow \bar{\bQ}_\ell$. Since the de Rham and \'etale cohomologies can be viewed as fiber functors of the category of motives with absolute Hodge cycles, we have an isomorphism of representations of $(G\MT)_{\bar{\bQ}_\ell}$:
$$H^1\dR(A/\bar{K})\otimes \bar{\bQ}_\ell\simeq H^1_\et(A_{\bar{K}},\bQ_\ell)\otimes \bar{\bQ}_\ell.$$
Via the above isomorphism, we denote by $V_i^\et$ the image of $V_i\otimes \bar{\bQ}_\ell$. Then by Faltings isogeny theorem, $V_i^\et$ are irreducible representations of $(G_\ell)_{\bar{\bQ}_\ell}$ and any two of them are isomorphic if and only if they are isomorphic as representations of $(G\MT)_{\bar{\bQ}_\ell}$. 

Now we study the action of the Frobenius torus $T_v^\circ$ on both sides. We use $(T_v^\circ)_{\bar{K}_v}$ to denote the base change of the crystalline one acting on the left hand side and use $(T_v^\circ)_{\bar{\bQ}_\ell}$ to denote the base change of the \'etale one acting on the right hand side. By \cite{K17}*{Cor. 2.3.1}, after raising both Frobenius actions to high enough power, $\varphi_v^{m_v}$ is conjugate to $Frob_v$ by an element of $G\MT$.\footnote{Recall that we use $Frob_v$ to denote the relative Frobenius action on the \'etale cohomology. Noot shows that after raising to a high enough power, there exists an element $g\in G\MT(\bQ)$ such that $g$ is conjugate to $\varphi_v^{m_v}$ by some element in $G\MT(\bar{K}_v)$ and that $g$ is conjugate to $Frob_v$ by some element in $G\MT(\bQ_\ell)$.} Then the weights of the action of $T_v^\circ$ on $V_i$ and $V_i^\et$ coincide and the fact that $V_i$ is isomorphic to $V_j$ as representations of $(T_v^\circ)_{\bar{K}_v}$ is equivalent to the fact that $V_i^\et$ is isomorphic to $V_j^\et$ as representations of $(T_v^\circ)_{\bar{\bQ}_\ell}$.
\end{para}

The following lemma reduces comparing the centralizers of $G\MT$ and $G^\circ_\aT$ to studying the irreducibility of $V_i$ as representations of $G^\circ_{\aT,\bar{K}}$.

\begin{lemma}\label{nonisom}
If $V_i$ and $V_j$ are not isomorphic as $G_{\mathrm{MT},\bar{K}}$-representations, then they are not isomorphic as $G^\circ_{\aT,\bar{K}}$-representations. In particular, if all $V_i$ are irreducible representations of $G^\circ_{\aT,\bar{K}}$, then $G^\circ_\aT$ and $G\MT$ have the same centralizer in $\End(H^1\dR(A/\bar{K}))$.
\end{lemma}

\begin{proof}
Let $v\in M_{\max}$ as in \Cref{Tvmax}. Then $T_v$ is connected and as a subgroup of $G_\ell$, it is a maximal torus. Since $V_i$ and $V_j$ are not isomorphic as representations of $G_{\mathrm{MT},\bar{K}}$, then by \ref{dRtoet}, their counterparts $V_i^\et$ and $V_j^\et$ are not isomorphic as representations of $(G_\ell)_{\bQ_\ell}$. Since $T_v$ is a maximal torus, then $V_i^\et$ and $V_j^\et$ are not isomorphic as representations of $(T_v)_{\bar{\bQ}_\ell}$. Then by the discussion in \ref{dRtoet}, $V_i$ and $V_j$ are not isomorphic as representations of $(T_v)_{\bar{K}_v}$. In particular, they are not isomorphic as representations of $G_\aT^\circ$.
\end{proof}

\begin{para}\label{gerb}
Let $v\notin \Sigma$ (defined in \cref{MumfordTate}) be a finite place of $K^{\aT}$ giving rise to an embedding $K^{\aT}\rightarrow K^{\aT}_v$.  
Since $G_\aT(\Res^K_\bQ A)$ is defined as the stabilizer of all absolute Tate cycles of $\Res^K_\bQ A$ and $\varphi_v(s_\alpha)=s_\alpha$ for all absolute Tate cycles $\{s_\alpha\}$, 
the action $\ad(\varphi_v)$, the conjugation by $\varphi_v$, is an isomorphism between $\sigma^*G_{\aT,K^{\aT}_v}(\Res^K_\bQ A)$ and $G_{\aT,K^{\aT}_v}(\Res^K_\bQ A)$, where $\sigma:K^{\aT}_v\rightarrow K^{\aT}_v$ is the Frobenius. The action $\ad(\varphi_v)$ also induces an isomorphism between the neutral connected components of these two groups. By taking the projection from $G_\aT(\Res^K_\bQ A)$ to $G_\aT(A)$, the action $\ad(\varphi_v)$ is an isomorphism between $\sigma^*G_{\aT,K^{\aT}_v}(A)$ and $G_{\aT,K^{\aT}_v}(A)$ and between their neutral connected components.
In terms of $K^{\aT}_v$-points, $\ad(\varphi_v)$ is a $\sigma$-linear automorphism of both $G_{\aT}(A)(K^{\aT}_v)$ and $G_{\aT}^\circ(A)(K^{\aT}_v)$.
Moreover, since all Hodge cycles are absolute Tate and hence fixed by $\varphi_v$, the above discussion also shows that $\ad(\varphi_v)$ defines a $\sigma$-linear automorphism on $G\MT(K_v^\aT)$.
\end{para}

\begin{proposition}\label{decomp}
Assume that $A_{\bar{K}}$ is simple. Then all $\rho_{\bar{K},i,j}$ are of the same dimension. Moreover, if the assumption in \ref{assumpA} holds and either a maximal subfield of $\End^\circ_{\bar{K}}(A)$ is Galois over $\bQ$ or $\End^\circ_{\bar{K}}(A)$ is a field, then there exists a choice of decomposition $\bigoplus_{i=1}^n V_i$ such that $\varphi_v(V_{i,j})=V_{\sigma(i),\tau_{v,i}(j)}$, where $\sigma$ is a permutation of $\{1,\dots, n\}$ and for each $i$, $\tau_{v,i}$ is a permutation of $\{1,\dots,n_i\}$.
\end{proposition}

\begin{proof}
The idea to prove the first assertion is to apply Bost's theorem after showing that $\phi_v$ preserves the vector space of the direct sum of all $V_{i,j}$ of a fixed dimension.
We fix a finite extension $L$ of $K$ such that all the $V_{i,j}$ are defined over $L$. Let $v\notin \Sigma$ be a place of $L$. 
As discussed in \cref{gerb}, $\ad(\varphi_v)$ preserves the set $G_{\aT}^\circ(L_v)$. Therefore, for any nonzero vector $v_{i,j}\in V_{i,j}$, as $\bQ_p$-linear spaces,
\begin{equation*}
\begin{split}
\varphi_v(V_{i,j})&=\varphi_v(\Span_{L_v}(G_{\aT}^\circ(L_v)(v_{i,j})))=\Span_{L_v}(\varphi_v(G_{\aT}^\circ(L_v)(v_{i,j})))\\
&=\Span_{L_v}(\ad(\varphi_v)(G_{\aT}^\circ(L_v)(\varphi_v(v_{i,j}))))=\Span_{L_v}(G_{\aT}^\circ(L_v)(\varphi_v(v_{i,j}))).
\end{split}
\end{equation*}

In other words, as an $L_v$-vector space, $\varphi_v(V_{i,j})$ is the same as the space of the irreducible $G^\circ_{\aT}(L_v)$-sub representation generated by $\varphi_v(v_{i,j})$. Similarly for $V_i$, we have that the vector space $\varphi_v(V_i)$ is the same as the vector space of an irreducible $G_{\mathrm{MT}}(L_v)$-sub representation. In particular, $\varphi_v(V_{i,j})$ is contained in $\bigoplus_{\dim V_{k,l}=\dim V_{i,j}}V_{k,l}$. Let $V'$ be $\bigoplus_{\dim V_{k,l}=\dim V_{i,j}}V_{k,l}$ and $V''$ be $\bigoplus_{\dim V_{k,l}\neq\dim V_{i,j}}V_{k,l}$. Then $\varphi_v(V')=V'$ and $\varphi_v(V'')=V''$. Let $pr'$ be the projection to $V'$. Then $\varphi_v(pr')=pr'$ for all $v\notin \Sigma$. By \Cref{Bost}, $pr'$ is an algebraic endomorphism of $A$. Since $A$ is simple, $pr'$ cannot be a nontrivial idempotent and then $V''=0$, which is the first assertion.

The second assertion is an immediate consequence of the following two lemmas.
Indeed, by \Cref{nomoresub}, we see that the only irreducible sub representations in $V_s$ of $G^\circ_{\aT}$ are $V_{s,j}$'s. Since $\varphi_v(V_{i,j})$ is the vector space of some sub representation of the $G_\aT^\circ$-representation with underlying vector space $\varphi_v(V_i)=V_s$ for some $s$ by \Cref{Fperm}, then $\varphi_v(V_{i,j})$ is $V_{s,t}$ for some $t$.
\end{proof}

\begin{lemma}\label{Fperm}
Under the assumptions in \Cref{decomp}, there exists a decomposition $H^1_{\mathrm{dR}}(A/L)=\bigoplus_{i=1}^n V_i$ where $V_i$ are irreducible representations of $G_{\mathrm{MT}}$ such that for any $i$, as vector spaces, $\varphi_v(V_i)=V_j$ for some $j$. 
\end{lemma}
\begin{proof}
When $\End^\circ(A)$ is a field, the decomposition is unique and any two distinct $V_j$'s are not isomorphic.
Since the vector space $\varphi_v(V_i)$ is the vector space of an irreducible sub representation, it must be $V_j$ for some $j$.

Now we assume that the maximal subfield of $\End^\circ_{\bar{K}}(A)$ is Galois over $\bQ$. Let $\{s_\alpha\}$ be a $\bQ$-basis of algebraic cycles in $\End(H^1_{\mathrm{dR}}(A/K^{\aT}))$ 
and we use $s_\alpha^B$ to denote their images in $\End(H^1_{\mathrm{B}}(A_\bC,\bQ))$. Since $A$ is simple, $\End^\circ_{\bar{K}}(A)$ is a division algebra $D$ of index $d$ over some field $F$ and $\{s^B_\alpha\}$ is a basis of $D$ as a $\bQ$-vector space. Let $E\subset D$ be a field of degree $d$ over $F$. Then $E$ is a maximal subfield of $D$ and $D\otimes_FE\cong M_d(E)$. Therefore, $$D\otimes_{\bQ}E\cong D\otimes_{\bQ}F\otimes_F E\cong D\otimes_F E\otimes_{\bQ}F\cong M_d(E)^{[F:\bQ]}.$$ Let $e_i\in M_d(E)$ be the projection to the $i$-th coordinate. Let $e_i^j\in D\otimes_{\bQ}E$ be the element whose image in $M_d(E)^{[F:\bQ]}$ is $(0,..,0,e_i,0,...,0),$ where $e_i$ is on the $j$-th component. Since $\sum_{i,j=1}^n e_i^j$ is the identity element in $D$, there must exist at least one $e_i^j$ such that $\sum_{\tau\in \Gal(E/\bQ)} \sigma(e_i^j)$ is nonzero, where the Galois action on $D\otimes_\bQ E$ is the one induced by the natural action on $E$.

We write $e_i^j=\sum k_\alpha s_\alpha^B$, where $k_\alpha\in E$, and let $pr_\tau=\sum \tau(k_\alpha) s_\alpha \in \End(H^1_{\mathrm{dR}}(A/\bar{K})),$ for all $\tau\in \Gal(E/\bQ)$. Since $e_i^j$ is an idempotent, so is $pr_\tau$. Let $V_\tau$ be the image of $pr_\tau$. We may assume that $L$ contains $E$ and still use $\sigma$ to denote the image of the Frobenius via the map $\Gal(L_v/\bQ_p)\subset \Gal(L/\bQ)\rightarrow \Gal(E/\bQ)$. Then by definition, $\varphi_v(pr_\tau)=pr_{\sigma\tau}$ and in other words, as vector spaces, $\varphi_v(V_\tau)=V_{\sigma\tau}$. 

Now it remains to prove that $\sum_{\tau\in \Gal(E/\bQ)} V_\tau$ is a direct sum and $\bigoplus_{\tau\in \Gal(E/\bQ)} V_\tau=H^1_{\mathrm{dR}}(A/\bar{K})$ as representations of $G\MT$. First, since $pr_\tau$ lies in the centralizer of $G\MT$, every $V_\tau$ is a subrepresentation. Second, since the number of irreducible representations in a decomposition of $H^1_{\mathrm{dR}}(A/\bar{K})$ equals $[E:\bQ]$, it suffices to prove that $\sum_{\tau\in \Gal(E/\bQ)} V_\tau=H^1_{\mathrm{dR}}(A/\bar{K})$. Since the image of the map $\sum_{\tau\in \Gal(E/\bQ)} pr_\tau$ is contained in $\sum_{\tau\in \Gal(E/\bQ)} V_\tau$, it suffices to prove that $\sum_{\tau\in \Gal(E/\bQ)} pr_\tau$ is invertible. By construction, the map $\sum_{\tau\in \Gal(E/\bQ)} pr_\tau$ is $\Gal(E/\bQ)$-equivalent and hence lies in $D$ (via comparison). Moreover, $\sum_{\tau\in \Gal(E/\bQ)} pr_\tau$ is nonzero by the choice of $e_i^j$ and hence this map is invertible since $D$ is a division algebra.
\end{proof}

\begin{lemma}\label{nomoresub}
The $G^\circ_{\aT}$-representations $V_{i,j}$ and $V_{i,j'}$ are not isomorphic if $j\neq j'$.
\end{lemma}
\begin{proof}
Let $T_w$ be a Frobenius torus of maximal rank for some finite place $w\notin \Sigma$. We also assume that $G\MT$, as a $\bQ$-group, is unramified at $p_w$. We only need to show that the weights of $T_w$ acting on $V_i$ are all distinct. By \ref{dRtoet}, we view $T_w$ as a maximal torus of $G_\ell(A)$ acting on irreducible sub representations of $(G_\ell)_{\bar{\bQ}_\ell}$ in $H^1_\et(A_{\bar{K}},\bQ_\ell)\otimes \bar{\bQ}_\ell$. \cite{Pink}*{Thm. 5.10} shows that $(G_\ell(A),\rho_\ell)$ is a weak Mumford--Tate pair over $\overline{\bQ}_\ell$. To show the weights on $V_i$ are distinct, it suffices to show that the weights of the maximal torus of any given geometrical irreducible component of $(G_\ell(A), \rho_\ell)$ are distinct. Furthermore, it reduces to the case of an almost simple component of each irreducible component. They are still weak Mumford--Tate pairs by \cite{Pink}*{4.1}. One checks the list of simple weak Mumford--Tate pairs in \cite{Pink}*{Table 4.2} to see that all the weights are distinct.
\end{proof}

\subsection{The statement of the theorem and its proof}\label{sub_pf}
\begin{para}\label{irryes}
Since the Mumford--Tate conjecture holds for the abelian varieties that we study in this subsection and especially in \Cref{thm_K},\footnote{The Mumford--Tate conjecture for absolutely simple abelian varieties is well-studied and we will cite the results that we need for each case in this subsection. The reduction of the conjecture for the product of simple abelian varieties to the conjecture for its simple factors is essentially contained in \cite{Lom}*{Sec. 4} and we record a proof at the end of this subsection.} we focus on using results in \cref{sub_gpth} to compare the centralizers of $G\MT$ and $G^\circ_\aT$ in $\End(H^1\dR(A/\bar{K}))$. 
Once we prove that the centralizers of both groups are the same, we conclude the proof of \Cref{thm_K} by \Cref{=rk} and \Cref{keylem} as in the proof of \Cref{Qcase}.
\end{para}
We separate the cases using Albert's classification.
\subsubsection*{Type I}
Let $F$ be the totally real field $\End^\circ_{\bar{K}}(A)$ of degree $e$ over $\bQ$. \cite{BGK1} shows that \Cref{conjMT} holds when $g/e$ is odd. 
\begin{proposition}\label{e=g}
If $e=g$, then \Cref{conj_main} holds for $A$.
\end{proposition}
\begin{proof}
The action of $F$ on the $\bQ$-vector space $H^1_{\mathrm{B}}(A,\bQ)$ makes it a $2$-dimensional $F$-vector space. Therefore, as a $(G\MT)_{\bar{\bQ}}$-representation, $H^1_{\mathrm{B}}(A,\overline{\bQ})$ decomposes into $g$ non-isomorphic irreducible sub representations of dimension $2$. Hence, the $G_{\mathrm{MT},\bar{K}}$-representation $H^1_{\mathrm{dR}}(A/\bar{K})$ decomposes into $g$ non-isomorphic irreducible sub representations $V_1,\dots, V_g$. By \Cref{nonisom} and the discussion in \ref{irryes}, we only need to show that all $V_i$ are irreducible $G^\circ_{\aT}$-representations. By \Cref{decomp}, if any $G^\circ_{\aT}$-representation $V_i$ is reducible, then all $V_1,\dots, V_g$ are reducible. In such situation, all $V_i$ decompose into $1$-dimensional representations and hence $G^\circ_{\aT}$ is a torus. Then by \Cref{cm=torus}, A has complex multiplication, which contradicts our assumption. 
\end{proof}
\begin{rem}\label{pdim}
The above proof is still valid if all (equivalently, any) $V_i$ are of prime dimension.
\end{rem}

Now we focus on the case when $\End_{\bar{K}}(A)=\bZ$. We refer the reader to \cite{Pink} for the study of \Cref{conjMT} in this case. In particular, \Cref{conjMT} holds when $2g$ is not of the form $a^{2b+1}$ or $\binom{4b+2}{2b+1}$, where $a,b\in \bN\backslash \{0\}$ and in this situation, $G_\ell(A)=\GSp_{2g,\bQ_\ell}$.
\begin{proposition}\label{EndZ}
Assume that $G_\ell(A)=\GSp_{2g,\bQ_\ell}$. If $A$ is defined over a number field $K$ which is Galois over $\bQ$ of degree $d$ prime to $g!$, then \Cref{conj_main} holds for $A$.
\end{proposition}
\begin{proof}
\Cref{conjMT} holds when $G_\ell(A)=\GSp_{2g}$. By \cref{irryes}, it suffices to show that $H^1_{\mathrm{dR}}(A/\bar{K})$ is an irreducible $G^\circ_{\aT,\bar{K}}$-representation. If not, then by \Cref{decomp}, $H^1_{\mathrm{dR}}(A/\bar{K})$ would decompose into $r$ sub representations of dimension $2g/r$ for some positive divisor $r$ of $2g$. By \Cref{cm=torus}, $r$ cannot be $2g$ and hence $r\leq g$. Let $pr^j$ be the projection to the $j$th irreducible component. By \Cref{decomp}, we have $\varphi_v(pr^j)=pr^k$ for some $k$ and the action of $\varphi_v$ on all $pr^j$ gives rise to an element $s_v$ in $S_r$, the permutation group on $r$ elements. On the other hand, by \Cref{TvQ}, there exists a set $M$ of rational primes of natural density $1$ such that for any $p\in M$ and any $v|p$, we have that $\varphi_v^{m_v}\in G^\circ_{\aT}$. Hence $s_v^{m_v}$ is the identity in $S_r$ for such $v$. By the assumption, $m_v$ is prime to $r!$ and hence $s_v$ is trivial in $S_r$. In other words, $pr^j$ is a $1$-absolute-Tate cycle. Then by \Cref{Bost}, $pr^j$ is algebraic, which contradicts with the fact that $A$ is simple.
\end{proof}

\subsubsection*{Type II and Type III} In this case, $\End^\circ_{\bar{K}}(A)$ is a quaternion algebra $D$ over a totally real field $F$ of degree $e$ over $\bQ$. \cite{BGK1} and \cite{BGK2} show that if $g/(2e)$ is odd and not in the set $\{\frac 12\binom{2^{m+2}}{2^{m+1}},m\in \bN\}$, then \Cref{conjMT} holds.\footnote{Victoria Cantoral-Farf\'an pointed out to me that for abelian varieties of type III, the proof in \cite{BGK2} is valid when $g/(2e)$ is odd and not in the set $\{\frac 12\binom{2^{m+2}}{2^{m+1}},m\in \bN\}$. See \cite{victoria}*{Thm.~5.2.1} for more details.}

\begin{proposition}\label{g=2e}
If $g=2e$, then \Cref{conj_main} holds for $A$.
\end{proposition}
\begin{proof}
The $G_{\mathrm{MT},\bar{K}}$-representation $H^1_{\mathrm{dR}}(A/\bar{K})$ decomposes into $V_1\oplus\cdots\oplus V_g$, where $V_i$ is $2$-dimensional and $V_i$ is not isomorphic to $V_j$ unless $\{i,j\}=\{2k-1,2k\}$. Then we conclude by \Cref{pdim}.
\end{proof}

\subsubsection*{Type IV} In this case, $\End^\circ_{\bar{K}}(A)$ is a division algebra $D$ over a CM field $F$. Let $[D:F]=d^2$ and $[F:\bQ]=e$. Then $ed^2|2g$. 

\begin{proposition}\label{IVcent}
If $\frac{2g}{ed}$ is a prime, then the centralizer of $G^\circ_{\aT}$ in $\End_{\bar{K}}(H^1_{\mathrm{dR}}(A/\bar{K}))$ is the same as that of $G\MT$.
\end{proposition}
\begin{proof}
We view $H^1_{\mathrm{B}}(A,\bQ)$ as a $F$-vector space and hence view $G\MT$ as a subgroup of $\GL_{2g/e}$. Since the centralizer of $G\MT$ is $D$, then $H^1_{\mathrm{B}}(A,\bQ)\otimes_F \bar{F}$ decomposes into $d$ representations of dimension $2g/(ed)$. Hence $H^1_{\mathrm{B}}(A,\bar{\bQ})$ as a $G\MT$-representation decomposes into $de$ representations of dimension $2g/(de)$. Then we conclude by \Cref{pdim}.
\end{proof}

\begin{corollary}\label{IV}
If $g$ is a prime, \Cref{conj_main} holds for $A$ of type IV.
\end{corollary}
\begin{proof}
Notice that when $g$ is a prime, then $d$ must be $1$ and $e$ must be $2$ or $2g$. The second case is when $A$ has complex multiplication and \Cref{conj_main} is known. In the first case, $\frac{2g}{ed}(=g)$ is a prime. Then \Cref{conj_main} is a consequence of \Cref{IVcent} and the Mumford--Tate conjecture (\cite{Chi1}*{Thm. 3.1}) by \Cref{keylem}.
\end{proof}

\begin{corollary}\label{primedim}
If the dimension of $A$ is a prime and $\End_{\bar{K}}(A)$ is not $\bZ$, then \Cref{conj_main} holds.
\end{corollary}
\begin{proof}
If $g=2$, then $A$ has CM or is of type I with $e=g$ or is of type II with $g=2e$. If $g$ is an odd prime, then $A$ is of type I with $e=g$ or of type IV. We conclude by \Cref{e=g}, \Cref{g=2e}, and \Cref{IV}.
\end{proof}

The following theorem summarizes the main results discussed above. 
\begin{theorem}\label{thm_K}
Let $A$ be a polarized abelian variety over some number field $K$ and assume that $A$ satisfies one of the following conditions:
\begin{enumerate}
\item $A$ is an elliptic curve or has complex multiplication.
\item $A_{\bar{K}}$ is simple; the dimension of $A$ is a prime number and $\End_{\bar{K}}(A)$ is not $\bZ$.
\item The polarized abelian variety $A$ of dimension $g$ with $\End_{\bar{K}}(A)=\bZ$ is defined over a finite Galois extension $K$ over $\bQ$  such that $[K:\bQ]$ is prime to $g!$ and $2g$ is not of form $a^{2b+1}$ or $\binom{4b+2}{2b+1}$ where $a,b\in \bN\backslash \{0\}$. 
\end{enumerate}
Then the absolute Tate cycles of $A$ coincide with its Hodge cycles.
\end{theorem}

\section{A strengthening of the result of Bost and its application}\label{density}
The main result of this section is \Cref{3/4} and the idea is to apply \Cref{Gas}. This result is a strengthening of \Cref{Bost}. At the end of this section, we use \Cref{3/4} to study cycles of abelian surfaces. For simplicity, when there is no risk of confusion, we use $\beta$-cycles to indicate $\beta$-absolute-Tate cycles.
\subsection{A refinement of a theorem of Gasbarri in a special case}
For simplicity, we only work with the higher dimensional Nevanlinna theory for $\bC^d$. It is contained in the classical theory developed by Griffiths and King \cite{GK}. See also \cite{B01}*{Sec. 4.3} and \cite{G}*{Sec. 5.24}. We first recall the setting of the algebraization theorem.

\begin{para}
Let $X$ be a geometrically irreducible quasi-projective variety of dimension $N$ over some number field $K$, $P$ a $K$-point of $X$, and $\widehat{X}_P$ the formal completion of $X$ at $P$.
Let $\widehat{V}$ be a smooth formal subvariety of $\widehat{X}_P$ of dimension $d$. 
We say $\widehat{V}$ is \emph{algebraic} if the smallest Zariski closed subset $Y$ of $X$ containing $P$ such that $\widehat{V}\subset \widehat{Y}_{P}$ has the same dimension as $\widehat{V}$. We say $\widehat{V}$ is \emph{Zariski dense} if $Y=X$.

For every complex embedding $\sigma:K\rightarrow \bC$, we assume that there exists a holomorphic morphism $\gamma_\sigma:\bC^d\rightarrow X_\sigma(\bC)$ which sends $0$ to $P_\sigma$ and maps the germ of $\bC^d$ at $0$ biholomorphically onto the germ $V^{an}_\sigma$ of $\widehat{V}$.

Let $z=(z_1,\dots,z_d)$ be the coordinate of $\bC^d$. The Hermitian norm $||z||$ on $\bC^d$ is given by $(|z_1|^2+\cdots+|z_d|^2)^{1/2}$. Let $\omega$ be the Kahler form on $\bC^d-\{0\}$ defined by $dd^c\log ||z||^2$. Then $\omega$ is the pull-back of the Fubini--Study metric on $\bP^{d-1}(\bC)$ via $\pi:\bC^d-\{0\}\rightarrow \bP^{d-1}(\bC)$.

To use the slope method introduced by Bost, we choose a flat projective scheme $\cX$ over $\Spec(\cO_K)$ such that $\overline{X}\colonequals \cX_K$ is some compactification of $X$. We also choose a relatively ample Hermitian line bundle $(\cL,\{||\cdot||_\sigma\}_\sigma)$ on $\cX$ over $\Spec(\cO_K)$. Let $\eta$ is the first Chern form (also known as the curvature form) of the fixed Hermitian ample line bundle $\cL|_{X_\sigma}$. 
We always assume that $\eta$ is positive, which is possible by a suitable choice of the Hermitian metric. 
\end{para}

\begin{defn}
The \emph{characteristic function} $T_\gs(r)$ is defined as follows:
\begin{equation*}
T_\gs(r)=\int_0^r\frac {dt}t \int_{B(t)}\gamma_\sigma^*\eta\wedge \omega^{d-1},
\end{equation*}
where $B(t)$ is the ball around $0$ of radius $t$ in $\bC^d$.
\end{defn}

\begin{defn}
The \emph{order} $\rs$ of $\gs$ is defined to be $\displaystyle \limsup_{r\rightarrow \infty}\frac{\log T_\gs(r)}{\log r}$. It is a standard fact that $\rs$ is independent of the choice of the Hermitian ample line bundle on $\overline{X}_\sigma$. (See for example \cite{G}*{Thm.~4.13(c) and Prop.~5.9}.) When $\rs$ is finite, $\gs$ is of order at most $\rs$ if and only if for any $\epsilon>0$, we have $T_\gs(r)<r^{\rs+\epsilon}$ for $r$ large enough. We denote by $\rho$ the maximum of all $\rs$ over all archimedean places $\sigma$.
\end{defn}

\begin{para}
Let $\cF$ be an involutive subbundle of the tangent bundle $T_X$ of $X$. We only focus on the case when $\widehat{V}$ is the formal leaf of $\cF$ passing through $P$. We may spread out $\cF$ and $X$ and assume that they are defined over $\cO_K[1/n]$ for some integer $n$. For a finite place $v$ with residue characteristic $p_v\nmid n$, the tangent bundle $T_{X_{k_v}}$ can be identified with the sheaf of derivations. Given a derivation $D$, then its $p_v$-th iterate $D^{p_v}$ is still a derivation; the map $D\mapsto D^{p_v}$ is called the \emph{$p_v$-th power map}. Let $M_{good}$ be the set of finite places $v$ of $K$ such that $\text{char } (k_v)\nmid n$ and that $\cF\otimes k_v$ is stable under $p_v$-th power map of derivations. Let $\alpha$ be the \emph{arithmetic density} of bad places:
$$\limsup_{x\rightarrow \infty}\left(\sum_{v|p_v\leq x, v\notin M_{good}}\frac{[L_v:\bQ_{p_v}]\log p_v}{p_v-1}\right)\left([L:\bQ]\sum_{p\leq x}\frac{\log p}{p-1}\right)^{-1}.$$
\end{para}

\begin{theorem}\label{Gas}
Assume that $\widehat{V}$ is a formal leaf that is Zariski dense in $X$, then either $\widehat{V}$ is algebraic (and $N=d$) or we have $$1\leq \frac N{N-d}\rho\alpha.$$
\end{theorem}

This is a refinement of a special case of \cite{G}*{Thm. 5.21}. We first use it to prove \cref{prop_density1}.
\begin{proof}[Proof of \Cref{prop_density1}]
The idea is due to Bost. We apply \Cref{Gas} to the formal leaf $\widehat{V}$ passing through identity of the involutive subbundle of the tangent bundle of $G$ generated by $W$ via translation. We take the uniformization map to be the exponential map $W(\bC)\rightarrow \Lie G(\bC)\rightarrow G(\bC)$. It is a standard fact that the order $\rho$ of this uniformization map is finite.\footnote{In \cite{BW}*{p. 112}, they summarized some results of Faltings and W\"utholz that enable us to show $\rho$ is $2$ by standard complex analytic arguments.} 
On the other hand, the assumptions on $W$ are equivalent to the assumption that the density $\alpha$ of bad primes is $0$. There would be a contradiction with \Cref{Gas} if $\widehat{V}$ is not algebraic.
\end{proof}

To get the better bound here, we use some ideas from \cite{H}.  
We recall the notation in \cite{B01}*{Sec. 4.2.1} and \cite{H}.

\begin{para}
For $D\in \bN$, let $E_D$ be the finitely generated projective $\cO_K$-module $\Gamma(\cX,\cL^D)$. For $n\in \bN$, let $V_n$ be the $n$-th infinitesimal neighborhood of $P$ in $\widehat{V}$ and let $V_{-1}$ be $\emptyset$. We define a decreasing filtration on $E_D$ as follows: for $i\in \bN$, let $E_D^i$ be the sub $\cO_K$-module of $E_D$ consisting of elements vanishing on $V_{i-1}$. We consider 
$$\phi_D^i:E_D^i\rightarrow \ker(\cL^{\otimes D}|_{V_i}\rightarrow \cL^{\otimes D}|_{V_{i-1}})\cong S^i(T_P \widehat{V})^\vee \otimes (\cL_P)^{\otimes D},$$
where the first map is evaluation on $V_i$ and $S^i$ denotes the $i$-th symmetric power. We will also use $\phi_D^i$ to denote its linear extension $E^i_D\otimes K\rightarrow S^i(T_P \widehat{V})^\vee \otimes (\cL_P)^{\otimes D}$.
To define the height $h(\phi^i_D)$, we need to specify the structure of the source and the target of $\phi^i_D$ as Hermitian vector bundles (over $\cO_K$) . Notice that the choice of $\cX$ gives rise to a projective $\cO_K$-module $\cT^\vee$ in $(T_P\widehat{V})^\vee$. More precisely, since $\cX$ is projective, there is a unique extension $\cP$ of $P$ over $\cO_K$, we take $\cT^\vee$ to be the image of $\cP^*\Omega_{\cX/\cO_K}$ in $(T_P\widehat{V})^\vee$. Moreover, $\cP^*\cL$ is a projective $\cO_K$-module in $\cL_P$. Then for any finite place $v$, we have a unique norm $||\cdot||_v$ on $E^i_D\otimes K$ (resp. $S^i(T_P \widehat{V})^\vee \otimes (\cL_P)^{\otimes D}$) such that for any element $s$, $||p_v^ms||_v\leq p_v^{-m[K_v:\bQ_p]}$ if and only if $s\in E^i_D$ (resp. $s\in S^i \cT^\vee \otimes (\cP^*\cL)^{\otimes D}$).
For an archimedean place $\sigma$, given the Hermitian norm on $\cL$, we equip $E^i_D\otimes K$ and $\cL_P$ with the supremum norm and the restriction norm. We fix a choice of Hermitian norm on $T_P\widehat{V}$ and then obtain the induced norm on $S^i \cT^\vee \otimes (\cP^*\cL)^{\otimes D}$.\footnote{To obtain the norm on $S^i\cT^\vee$, we view it as a quotient of $(\cT^\vee)^{\otimes i}$.}
We define $$h(\phi^i_D)=\frac 1{[K:\bQ]}\sum_{\text{all places }v}h_v(\phi^i_D),\text{ where }h_v(\phi^i_D)=\sup_{s\in E^i_D, ||s||_v\leq 1}\log||\phi^i_D(s)||_v.$$
\end{para}

\begin{lemma}\label{infinite}
For any $\epsilon>0$, and any complex embedding $\sigma$, there exists a constant $C_1$ independent of $i,D$ such that
$$h_\sigma(\phi^i_D)\leq C_1(i+D)-\frac i {\rs+\epsilon} \log \frac i D.$$
In particular, $$\frac 1 {[K:\bQ]} \sum_\sigma h_\sigma(\phi^i_D)\leq C_1(i+D)-\frac i {\rho+\epsilon} \log \frac i D.$$
\end{lemma}
\begin{proof}
This is \cite{G}*{Thm. 5.19 and Prop. 5.26}. We sketch a more direct proof\footnote{We use the definition of the order as in \cite{B01} rather than as in \cite{G}. Gasbarri gave a proof showing that two definitions are the same, but in this paper, we only need to work with the definition in \cite{B01}.}  for the special case here using the same idea originally due to Bost. See also \cite{H}*{Lem. 6.8}.

By \cite{B01}*{Cor. 4.~16} (a consequence of the First Main Theorem in Nevanlinna theory), there exists a constant $B_1$ only depend on $d$ such that
$$h_\sigma(\phi^i_D)\leq -i\log r +D T_\gs(r) +B_1 i.$$
By the definition of $\rs$, there exists a constant $M>0$ such that for all $r>M$, we have $T_\gs(r)<r^{\rs+\epsilon}$.
On the other hand, as in the proof of \cite{G}*{Thm. 4.15}, $-i\log r+Dr^{\rs+\epsilon}$, as a function of $r$, reaches its minimum in $r_0=(\frac i {(\rs+\epsilon)D})^{1/(\rs+\epsilon)}$. Pick a constant $M_1$ such that if $i/D>M_1$, then $r_0>M$. Therefore, for $i, D$ such that $i/D>M_1$, we have
$$h_\sigma(\phi^i_D)\leq -i\log r_0 +Dr_0^{\rs+\epsilon} +B_1 i\leq -\frac i {\rs+\epsilon} \log \frac i D+B_2 i,$$
for some constant $B_2$.
Moreover, there exists a constant $B_3$ such that (see for example \cite{B01}*{Prop. 4.12}) for any $i,D$,
$$h_\sigma(\phi^i_D)\leq B_3(i+D).$$
Since there exists a constant $B_4$ such that $\frac i {\rs+\epsilon} \log \frac i D\leq B_4i$ for all $i, D$ such that $i/D\leq M_1$, we have 
$$h_\sigma(\phi^i_D)\leq (B_3+B_4)(i+D)-\frac i {\rs+\epsilon} \log \frac i D.$$
We can take $C_1$ to be $\max\{B_2, B_3+B_4\}$.
\end{proof}

\begin{lemma}[\cite{H}*{Prop. 3.6}]\label{finite}
For any $\epsilon>0$, there exists a constant $C_2$ such that
$$\frac 1 {[K:\bQ]} \sum_{\text{all finite places $v$}} h_v(\phi^i_D)\leq  (\alpha+\epsilon)i\log i+C_2(i+D).$$
\end{lemma}
\begin{proof}
We sketch the proof by Herblot for the convenience of readers. By \cite{B01}*{Lem.~4.10}, for any finite place $v$, there exists a constant $C_v$ such that $h_v(\phi^i_D)\leq C_v(i+D)$.
Let $M_{br}$ be the set of primes $v$ such that $\cX$ and $\widehat{V}$ are not smooth at $v$. By the discussion in \cite{B01}*{sec.~3.4.1}, for any $v\notin M_{br}$, if $p_v>i$, then $h_v(\phi^i_D)\leq 0$. Moreover, if $v\in M_{good}$, then $\displaystyle h_v(\phi^i_D)\leq \frac{i[K_v:\bQ_{p_v}]\log p_v}{p_v(p_v-1)}$ by \cite{CL}*{Lem.~7.6}; if $v\notin M_{good}\cup M_{br}$, then $\displaystyle h_v(\phi^i_D)\leq \frac{i[K_v:\bQ_{p_v}]\log p_v}{p_v-1}$ by \cite{B01}*{Prop.~3.9(1)}. Since $\displaystyle \sum_{p}\frac{\log p}{p(p-1)}<\infty$, then by putting the above inequalities together, we have (here $B_5$ is a constant)
$$\sum_{v\nmid \infty}h_v(\phi^i_D)\leq \sum_{v\in M_{br}}h_v(\phi^i_D)+\sum_{v\notin M_{br},p_v\leq i}h_v(\phi^i_D)\leq B_5(i+D)+\sum_{v\notin M_{good}\cup M_{br},p_v\leq i}\frac{i[K_v:\bQ_{p_v}]\log p_v}{p_v-1}.$$ Then the assertion follows from our defintion of $\alpha$ and the fact that $\displaystyle \sum_{p\leq i}\frac{\log p}{p-1}$ is asymptotic to $\log i$.
\end{proof}

\begin{proof}[Proof of \Cref{Gas}] We follow \cite{H}*{Sec. 6.6}.
By the slope inequality and the arithmetic Hilbert--Samuel theorem (\cite{B01}*{Sec. 4.1, 4.2}), we have (see \cite{H}*{Sec. 6.3, Eqn. (6.10)} for reference\footnote{Although Herblot focuses on the case when $d=1$, this inequality holds in general as all the results by Bost cited here hold in general.})
$$-C_3D^{N+1}\leq \sum_{i=0}^\infty rk(E_D^i/E_D^{i+1})(C_4(i+D)+h(\phi_D^i)).$$
By \Cref{infinite} and \Cref{finite}, we have
$$-C_3D^{N+1}\leq \sum_{i=0}^\infty rk(E_D^i/E_D^{i+1})(C_5(i+D)+(\alpha+\epsilon-\frac 1 {\rho+\epsilon})i\log i+\frac i {\rho+\epsilon} \log D).$$
Write $$S_D(\delta)=\sum_{i\leq D^\delta}rk(E_D^i/E_D^{i+1})(-C_5(i+D)+(-\alpha-\epsilon+\frac 1 {\rho+\epsilon})i\log i-\frac i {\rho+\epsilon} \log D),$$ $$S'_D(\delta)=\sum_{i> D^\delta}rk(E_D^i/E_D^{i+1})(-C_5(i+D)+(-\alpha-\epsilon+\frac 1 {\rho+\epsilon})i\log i-\frac i {\rho+\epsilon} \log D).$$
By \cite{B01}*{Lem. 4.7 (1)}, $rk(E^0_D/E_D^{i+1})<(i+1)^d$. Hence (see \cite{H}*{Lem. 6.14}) if $\delta\geq 1$, then
$$|S_D(\delta)|\leq C_6 D^\delta \log D \sum_{i\leq D^\delta}rk(E_D^i/E_D^{i+1})\leq C_7 D^{(d+1)\delta}\log D.$$
On the other hand, if $\frac 1 {1-(\rho+\epsilon)(\alpha+\epsilon)}<\delta<N/d$, then by \cite{B01}*{Lem.~4.7 (2)}, there exist constants $C_{10}, C_{11}$ such that for $D$ large enough, $$\sum_{i>D^\delta}rk(E^i_D/E^{i+1}_D)=rk(E^0_D)-rk(E^0_D/E_D^{[D^\delta]+1})\geq C_{10}D^N-([D^\delta]+1)^d\geq C_{11}D^N.$$ Then \cite{H}*{Lem. 6.15} shows that for $D$ large enough,
$$S'_D(\delta)\geq C_8 D^{N+\delta}\log D.$$
If there exists a $\delta$ such that $1\leq\frac 1 {1-(\rho+\epsilon)(\alpha+\epsilon)}<\delta<N/d$, then $$S'_D(\delta)+S_D(\delta)\geq C_9 D^{N+\delta}\log D$$ for $D$ large enough, which contradicts the fact that $S'_D(\delta)+S_D(\delta)\leq C_3 D^{N+1}$.
In other words, $N/d\leq \frac 1 {1-(\rho+\epsilon)(\alpha+\epsilon)}$. As $\epsilon$ is arbitrary, we obtain the desired result by rearranging the inequality.
\end{proof}

\subsection{A strengthening of \Cref{Bost}}
\begin{para}\label{sbar}
If $s\in \End(H^1_{\mathrm{dR}}(A/\bar{K}))$ is a $\beta$-cycles for some $\beta>0$, then $\varphi_v(s)=s$ for infinitely many $v$ and thus $s\in \Fil^0(\End(H^1_{\mathrm{dR}}(A/\bar{K})))$ by \Cref{Fil0}. In other words, $s$ maps $\Fil^1(H^1_{\mathrm{dR}}(A/\bar{K}))$ to itself. Since $$H^1_{\mathrm{dR}}(A/\bar{K})/\Fil^1(H^1_{\mathrm{dR}}(A/\bar{K}))\cong \Lie \dA_{\bar{K}},$$ the cycle $s$ then induces an endomorphism $\bar{s}$ of $\Lie \dA_{\bar{K}}$.
\end{para}

\begin{theorem}\label{w3/4}
Assume that $A_{\bar{K}}$ is simple. If $s\in \End(H^1_{\mathrm{dR}}(A/\bar{K}))$ is a $\beta$-absolute-Tate cycle for some $\beta>\frac 34$, then $\bar{s}$ is the image of some element in $\End^\circ_{\bar{K}}(\dA)$.
\end{theorem}

Before proving the theorem, we use it to prove a strengthening of \Cref{Bost}.

\begin{corollary}\label{3/4}
Assume that $A_{\bar{K}}$ is simple. If $s\in \End(H^1_{\mathrm{dR}}(A/\bar{K}))$ is a $\beta$-absolute-Tate cycle for some $\beta>\frac 34$, then $s$ is algebraic.
\end{corollary}

\begin{proof}
By \Cref{w3/4}, it suffices to show that if $s$ is fixed by infinitely many $\varphi_v$ and $\bar{s}$ (defined in \cref{sbar}) is algebraic, then $s$ is algebraic. Since the restriction to $\End^\circ_{L}(A)$ of the map
$$\Fil^0\End(H^1\dR(A/L))\rightarrow \End(\Lie A^\vee_L), \,s\rightarrow \bar{s}$$
is induced from the natural identification $\End^\circ_{L}(A)\cong\End^\circ_{L}(\dA)$, we obtain an algebraic cycle $t\in \End^\circ_{\bar{K}}(A)$ such that $\bar{t}=\bar{s}$. Then $s-t \in \Fil^1(\End(H^1_{\mathrm{dR}}(A/\bar{K}))$. Moreover, $s-t$ is a $\beta$-cycle and hence for infinitely many $v$, we have $\varphi_v(s-t)=s-t$. Then by \Cref{Fil0}, $s-t=0$ and hence $s$ is algebraic.
\end{proof}

\begin{rem}\label{g3/4}
The only place where we use the assumption of $A_{\bar{K}}$ being simple is to obtain the first assertion of \Cref{N=2g}. It says that if $s$ is not algebraic, then the Zariski closure of the $g$-dimensional formal subvariety that we will construct using $s$ is of dimension $2g$. In general, the Zariski closure of a non-algebraic $g$-dimensional formal subvariety is of dimension at least $g+1$ and then the same argument as below shows that any $\beta$-cycle with $\beta>1-\frac1{2(g+1)}$ is algebraic.
\end{rem}

\begin{para}
The proof of this theorem will occupy the rest of this subsection. Since the definition of $\beta$-cycle is independent of the choice of a field of definition and the property of being a $\beta$-cycle is preserved under isogeny, we may assume that $A$ is principally polarized and $s$ is defined over $K$.
 Let $X$ be $\dA\times \dA$ and $e$ be its identity. The main idea is to apply \Cref{Gas} to the following formal subvariety $\widehat{V}\subset \widehat{X}_{e}$.
Consider the sub Lie algebra ($\bar{s}$ is defined in \cref{sbar}) $$H=\{(a,\bar{s}(a))\mid a\in \Lie(\dA)\} \subset \Lie(X).$$ This sub Lie algebra induces an involutive subbundle $\cH$ of the tangent bundle of $X$ via translation. The formal subvariety $\widehat{V}$ is defined to be the formal leaf passing through $e$. A finite place $v$ of $K$ is called \emph{bad} if $\cH\otimes k_v$ is not stable under $p_v$-th power map of derivations.
\end{para}

\begin{lemma}\label{N=2g}
If $\bar{s}$ is not algebraic, then the formal subvariety $\widehat{V}$ is Zariski dense in $X$. The arithmetic density of bad primes is at most $1-\beta$.
\end{lemma}
\begin{proof}
The Zariski closure $G$ of $\widehat{V}$ must be an algebraic subgroup of $X$. The simplicity of $A$ implies that the only algebraic subgroup of $X$ with dimension larger than $g$ must be $X$. Hence if $\bar{s}$ is not algebraic, we have $\dim G>g$ and hence $G=X$. 

Let $E(\dA)$ denote the universal vector extension of $\dA$. The algebraic group $E(\dA)$ is commutative and $\Lie(E(\dA))=H^1_{\rm{dR}}(A/K)$. If $A$ has good reduction at $v$, then this identification extends to an identification of $\cO_v$-modules.
By \cite{Mum}*{p. 138}, given $v\notin \Sigma$, the $p_v$-th power map on $\Lie E(\dA)\otimes k_v=H^1_{\mathrm{dR}}(A/k_v)$ is the same as $\varphi_v\otimes k_v$. Therefore, for those $v$ such that $\varphi_v(s)=s$, we have that the Lie subalgebra $\{(a,s(a))\mid a\in (\Lie E(\dA))\}\otimes k_v$ of $\Lie(E(\dA)\times E(\dA))\otimes k_v$ is closed under the $p_v$-th power map. Then $H=\{(a,\bar{s}(a))\mid a\in \Lie(\dA)\}\otimes k_v$ and $\cH\otimes k_v$ are closed under the $p_v$-th power map.  Therefore, the density of bad primes is at most one minus the density of primes satisfying $\varphi_v(s)=s$.
\end{proof}

\begin{para}
Let $\sigma:K\rightarrow \bC$ be an archimedean place of $K$.
We define $\gs$ to be the composition 
$$\gs: \bC^g\xrightarrow{(id,\bar{s})}\bC^g\times \bC^g\xrightarrow{(\exp,\exp)}X_\sigma(\bC),$$ where $\exp$ the uniformization of $\bC^g=\Lie\dA_\sigma(\bC) \rightarrow \dA_\sigma(\bC)$. We choose an ample Hermitian line bundle $\cL$ on $\dA$ such that the pull back of its first Chern form via $\exp$ is $iC_0\sum_{k=1}^{g} dz_k\wedge d\bar{z}_k$ where $C_0>0$ is some constant. More explicitly, we may choose $\cL$ to be the theta line bundle with a translation-invariant metric. See for example \cite{RdJ}*{Sec. 2}. 

To compute the order of $\gs$, we fix the ample Hermitian line bundle on $X$ to be $pr_1^*\cL\otimes pr_2^*\cL$. Then $$\gs^*\eta=C_0(i\sum_{k=1}^{g} dz_k\wedge d\bar{z}_k+s^*(i\sum_{k=1}^{g} dz_k\wedge d\bar{z}_k)).$$ Thus $\gs^*\eta$ has all coefficients of $dz_i\wedge d\bar{z}_j$ being constant functions on $\bC^g$. 
\end{para}

\begin{lemma}
The order $\rs$ of $\gs$ is at most $2$. In other words, $\rho\leq 2$.
\end{lemma}
\begin{proof}
Up to a positive constant, 
$$\omega=i\frac{||z||^2\sum_{k=1}^g dz_k\wedge d\bar{z}_k-\sum_{k,l=1}^g \bar{z}_kz_l\,dz_k\wedge d\bar{z}_l}{||z||^4}.$$
Since all the absolute values of the coefficients of $dz_k\wedge d\bar{z}_l$ in $\omega$ are bounded by $2||z||^{-2}$ and those in $\gs^*\eta$ are constant functions, the volume form $\gs^*\eta\wedge \omega^{g-1}$ has the absolute value of the coefficient of $\wedge_{k=1}^g(dz_k\wedge d\bar{z}_k)$ to be bounded by $C_1||z||^{-2(g-1)}$ for some constant $C_1$. Hence
\begin{equation*}
\begin{split}
T_\gs(r)&=\int_0^r\frac {dt}t\int_{B(t)}\gs^*\eta\wedge \omega^{g-1}\\
&\leq \int_0^r\frac {dt}t\int_{B(t)}C_1||z||^{-2(g-1)}(i^g)\wedge_{k=1}^g(dz_k\wedge d\bar{z}_k)\\
&=\int_0^r\frac {dt}t\int_0^t C_2 R^{-2(g-1)} vol(S(R)) dR\\
&=\int_0^r\frac {dt}t\int_0^t C_3 RdR=C_4r^2,
\end{split}
\end{equation*}
where $S(R)$ is the sphere of radius $R$ in $\bC^d$. We conclude by the definition of orders.
\end{proof}

\begin{rem}
By a more careful argument, one can see that $\rs$ is $2$.
\end{rem}

\begin{proof}[Proof of \Cref{3/4}]
If $\bar{s}$ is not algebraic, then we apply \Cref{Gas} with $N=2g$ and $d=g$. We have
$$1\leq 2\rho\alpha\leq 2\cdot 2\cdot (1-\beta),$$ which contradicts with $\beta>\frac34$.
\end{proof}

\subsection{Abelian surfaces}
We see from the discussion in \cref{connected} that the only case left for \Cref{conj_main} for abelian surfaces is when $\End_{\bar{K}}(A)=\bZ$ and $K$ is of even degree over $\bQ$. We discuss in this section the case when $K$ is a quadratic extension of $\bQ$ and remark that one can deduce similar results when $[K:\bQ]$ is $2n$ for some odd integer $n$ by incorporating arguments as in the proof of \Cref{EndZ}.

Assume that $A$ does not satisfy \Cref{conj_main}. Then by \Cref{decomp}, we have $H^1_{\mathrm{dR}}(A/\bar{K})=V_1\oplus V_2$ as a $G^\circ_{\aT}$-representation, where $V_1$ and $V_2$ are irreducible representations of dimension $2$. 

Let $\ub$ be the inferior density of good primes of $pr_1$ $$\displaystyle \liminf_{x\rightarrow \infty}\left(\sum_{v|p_v\leq x, \varphi_v(pr_1)=pr_1}\frac{[L_v:\bQ_{p_v}]\log p_v}{p_v-1}\right)\left([L:\bQ]\sum_{p\leq x}\frac{\log p}{p-1}\right)^{-1}$$ and $\ob$ be the superior density $$\displaystyle\limsup_{x\rightarrow \infty}\left(\sum_{v|p_v\leq x, \varphi_v(pr_1)=pr_1}\frac{[L_v:\bQ_{p_v}]\log p_v}{p_v-1}\right)\left([L:\bQ]\sum_{p\leq x}\frac{\log p}{p-1}\right)^{-1}.$$ By \Cref{Tvmax}, for a density one set of split primes $v$ of $K$, we have $\varphi_v\in G^\circ_\aT(K_v)$ and then $\varphi_v(pr_i)=pr_i$ for $i=1,2$. In other words, we have $\frac 12\leq \ub\leq \ob$.

\begin{theorem}
If $A$ does not satisfy \Cref{conj_main}, then $\ub\leq \frac34\leq \ob$. In particular, if the natural density of good primes of $pr_1$ exists, then the density must be $\frac 34$.
\end{theorem}
\begin{proof}
By definition, $pr_1$ is a $\ub$-absolute-Tate cycle. If $\ub>\frac 34$, then by \Cref{3/4}, we have $pr_1$ is algebraic. As $\End_{\bar{K}}(A)=\bZ$, this is a contradiction. Therefore, $\ub\leq \frac 34$.

Let $\theta\in K$ be an element such that $\sigma(\theta)=-\theta$, where $\sigma$ is the nontrivial element in $\Gal(K/\bQ)$. We consider $\theta pr_1-\theta pr_2\in \End(H^1_{\mathrm{dR}}(A/\bar{K}))$. By \Cref{decomp}, if $\varphi_v(pr_1)\neq pr_1$, then $\varphi_v(pr_1)=pr_2$ and $\varphi_v(pr_2)=pr_1$. By the $\sigma$-linearity of $\varphi_v$, when $v$ is inert, we have that if $\varphi_v(pr_1)\neq pr_1$, and then $\varphi_v(\theta pr_1-\theta pr_2)=\theta pr_1-\theta pr_2$. For $v$ split, we have $\varphi_v(pr_i)=pr_i$ and hence $\varphi_v(\theta pr_1-\theta pr_2)=\theta pr_1-\theta pr_2$. Then by definition, $\theta pr_1-\theta pr_2$ is a $(\frac 32-\ob)$-absolute-Tate cycle. By \Cref{3/4}, if $\ob<\frac 34$, then $\theta pr_1-\theta pr_2$ is algebraic, which contradicts the assumption that $\End_{\bar{K}}(A)=\bZ$.
\end{proof}

\end{document}